\documentclass[final,12pt]{elsarticle}
\usepackage{srcltx}
\usepackage{eurosym}
\usepackage{mathtools}
\usepackage{amsmath}
\usepackage{amsfonts}
\usepackage{amssymb}
\usepackage{amsthm}
\usepackage{graphicx}
\usepackage{mathrsfs}
\usepackage[dvipsnames]{xcolor}
\usepackage{exscale}
\usepackage{latexsym}

\usepackage[colorlinks,plainpages=true,pdfpagelabels,hypertexnames=true,colorlinks=true,pdfstartview=FitV,linkcolor=blue,citecolor=red,urlcolor=black]{hyperref}
\PassOptionsToPackage{unicode}{hyperref}
\PassOptionsToPackage{naturalnames}{hyperref}
\usepackage{enumerate}
\usepackage[shortlabels]{enumitem}
\usepackage{bookmark}
\usepackage{wasysym}
\usepackage{esint}
\usepackage[ddmmyyyy]{datetime}
\usepackage[margin=1.85cm]{geometry}
\makeatletter
\g@addto@macro\normalsize{%
	\setlength\abovedisplayskip{4pt}
	\setlength\belowdisplayskip{4pt}
	\setlength\abovedisplayshortskip{4pt}
	\setlength\belowdisplayshortskip{4pt}
}
\numberwithin{equation}{section}
\everymath{\displaystyle}
\usepackage[capitalize,nameinlink]{cleveref}
\crefname{section}{Section}{Sections}
\crefname{subsection}{Subsection}{Subsections}
\crefname{condition}{Condition}{Conditions}
\crefname{hypothesis}{Hypothesis}{Conditions}
\crefname{assumption}{Assumption}{Assumptions}
\crefname{lemma}{Lemma}{Lemmas}
\crefname{definition}{Definition}{Definitions}

\crefformat{equation}{\textup{#2(#1)#3}}
\crefrangeformat{equation}{\textup{#3(#1)#4--#5(#2)#6}}
\crefmultiformat{equation}{\textup{#2(#1)#3}}{ and \textup{#2(#1)#3}}
{, \textup{#2(#1)#3}}{, and \textup{#2(#1)#3}}
\crefrangemultiformat{equation}{\textup{#3(#1)#4--#5(#2)#6}}%
{ and \textup{#3(#1)#4--#5(#2)#6}}{, \textup{#3(#1)#4--#5(#2)#6}}%
{, and \textup{#3(#1)#4--#5(#2)#6}}

\Crefformat{equation}{#2Equation~\textup{(#1)}#3}
\Crefrangeformat{equation}{Equations~\textup{#3(#1)#4--#5(#2)#6}}
\Crefmultiformat{equation}{Equations~\textup{#2(#1)#3}}{ and \textup{#2(#1)#3}}
{, \textup{#2(#1)#3}}{, and \textup{#2(#1)#3}}
\Crefrangemultiformat{equation}{Equations~\textup{#3(#1)#4--#5(#2)#6}}%
{ and \textup{#3(#1)#4--#5(#2)#6}}{, \textup{#3(#1)#4--#5(#2)#6}}%
{, and \textup{#3(#1)#4--#5(#2)#6}}

\crefdefaultlabelformat{#2\textup{#1}#3}
\numberwithin{equation}{section}

\newtheorem{theorem} {Theorem}[section]
\newtheorem{proposition}[theorem]{Proposition}
\newtheorem{lemma}[theorem]{Lemma}
\newtheorem{corollary}[theorem]{Corollary}

\newtheorem{counter example}[theorem]{Counter Example}
\newtheorem{remark}[theorem] {Remark}
\newtheorem{definition}[theorem] {Definition}

\newtheorem{claim}[theorem] {Claim}

\def\CC{{\rm \kern.24em \vrule width.02em height1.4ex depth-.05ex \kern-.26emC}}

\def\TagOnRight

\def\AA{{it I} \hskip-3pt{\tt A}}

\def\QQ{\rlap {\raise 0.4ex \hbox{$\scriptscriptstyle |$}} {\hskip -0.1em Q}}

\makeatletter
\newcommand{\vo}{\vec{o}\@ifnextchar{^}{\,}{}}
\makeatother

\def\YYint#1#2#3{{\setbox0=\hbox{$#1{#2#3}{\iint}$}
		\vcenter{\hbox{$#2#3$}}\kern-.50\wd0}}

\def\XXint#1#2#3{{\setbox0=\hbox{$#1{#2#3}{\int}$}
		\vcenter{\hbox{$#2#3$}}\kern-.50\wd0}}

\makeatletter
\def\namedlabel#1#2{\begingroup
	\def\@currentlabel{#2}%
	\label{#1}\endgroup
}
\makeatother
\makeatletter
\newcommand{\rmh}[1]{\mathpalette{\raisem@th{#1}}}
\newcommand{\raisem@th}[3]{\hspace*{-1pt}\raisebox{#1}{$#2#3$}}
\makeatother

\newcounter{desccount}

\newcommand{\descref}[2]{\hyperref[#1]{\textnormal{\textcolor{black}{}\textcolor{blue}{\bf #2}\textcolor{black}{}}}}

\newcommand{\dref}[2]{\hyperref[#1]{\textcolor{black}{(}\textcolor{blue}{\bf #2}\textcolor{black}{)}}}
\newcommand{\be} {\begin{eqnarray}}
\newcommand{\ee} {\end{eqnarray}}
\newcommand{\Bea} {\begin{eqnarray*}}
	\newcommand{\Eea} {\end{eqnarray*}}
\newcommand{\pa} {\partial}
\newcommand{\re}{\mathbb{R}}

\newcommand{\al} {\alpha}
\newcommand{\rr}{\rightarrow}

\newcommand{\dip}{\displaystyle}
\newcommand{\B} {\beta}
\newcommand{\de} {\delta}
\newcommand{\g} {\gamma}

\newcommand{\p}  {\prime}
\newcommand{\e}  {\epsilon}

\newcommand{\si} {\sigma}

\newcommand{\f}{\infty}
\newcommand{\R}{\mathbb{R}}

\newcommand{\ep}{\epsilon}

\newcommand{\lla} {\left\langle}
\newcommand{\rra} {\right\rangle}

\newcommand{\tm}{\bar{\textbf{m}}}
\newcommand{\tT}{\mathfrak{T}}
\newcommand{\tr}{\mathfrak{r}}
\newcommand{\mr}{\mathfrak{r}_\e}
\newcommand{\mv}{\mathfrak{u}_\e}
\newcommand{\mT}{\mathfrak{T}_\e}
\newcommand{\tS}{\bar{\mathcal{S}}}
\newcommand{\td}{\bar{\rho}}
\newcommand{\tms}{{\textbf{m}}}
\newcommand{\tSs}{{\mathcal{S}}}
\newcommand{\tds}{{\rho}}
\newcommand{\tQ}{\mathscr{Q}}
\newcommand{\tP}{\mathcal{P}\left(\tQ\right)}
\newcommand{\me}{\mathcal{M}^+\left(\bar{\Omega}\right)}
\newcommand{\tme}{\mathcal{M}^+\left(S^{N-1}\times\bar{\Omega}\right)}
\newcommand{\tY}{\mathcal{Y}}
\newcommand{\tDkin}{\mathfrak{D}_{\mbox{kin}}}
\newcommand{\tDint}{\mathfrak{D}_{\mbox{int}}}
\newcommand{\tDcon}{\mathfrak{D}_{\mbox{conv}}}
\newcommand{\dom}{\int\limits_{\si}^{\tau}\int\limits_{\Omega}} 
\newcommand{\domm}{\int\limits_{\si}^{\tau}\int\limits_{\Omega}}                      

\newcommand{\vt}{\vartheta}

\newcommand{\Ov}[1]{\overline{#1}}

\newcommand{\abs}[1]{\left| #1\right|}

\newcounter{whitney}
\refstepcounter{whitney}

\newcounter{ineqcounter}
\refstepcounter{ineqcounter}
\makeatletter
\def\ps@pprintTitle{%
	\let\@oddhead\@empty
	\let\@evenhead\@empty
	\def\@oddfoot{}%
	\let\@evenfoot\@oddfoot}
\makeatother
\usepackage[doublespacing]{setspace}
\usepackage[titletoc,toc,page]{appendix}

\makeatletter
\newcommand{\refcheckize}[1]{%
	\expandafter\let\csname @@\string#1\endcsname#1%
	\expandafter\DeclareRobustCommand\csname relax\string#1\endcsname[1]{%
		\csname @@\string#1\endcsname{##1}\wrtusdrf{##1}}%
	\expandafter\let\expandafter#1\csname relax\string#1\endcsname
}
\makeatother

\refcheckize{\cref}
\refcheckize{\Cref}


\makeatletter
\newcommand{\mainsectionstyle}{%
	\renewcommand{\@secnumfont}{\bfseries}
	\renewcommand\section{\@startsection{section}{2}%
		\z@{.5\linespacing\@plus.7\linespacing}{-.5em}%
		{\normalfont\bfseries}}%
}
\makeatother
\usepackage{pgf,tikz}
\usetikzlibrary{arrows}
\usetikzlibrary{decorations.pathreplacing}

\usepackage{xpatch}
\xpatchcmd{\MaketitleBox}{\hrule}{}{}{}
\xpatchcmd{\MaketitleBox}{\hrule}{}{}{}

\date{}
\begin{document}
	
	\begin{frontmatter}
		
		\title{Uniqueness of  dissipative solutions to the complete Euler system 
		} 
		\author[myaddress]{Shyam Sundar Ghoshal}
		\ead{ghoshal@tifrbng.res.in}

		\author[myaddress]{Animesh Jana}
		\ead{animesh@tifrbng.res.in }
		
		\address[myaddress]{Centre for Applicable Mathematics,Tata Institute of Fundamental Research, Post Bag No 6503, Sharadanagar, Bangalore - 560065, India.}
		\vspace{-2ex}
		\begin{abstract}
			Dissipative solutions have recently been studied as a generalized concept for weak solutions of the complete Euler system. Apparently, these are expectations of suitable measure valued solutions. 
			Motivated from [Feireisl, Ghoshal and Jana, Commun. Partial Differ. Equ., 2019], we impose a one-sided Lipschitz bound on velocity component as uniqueness criteria for a weak solution in Besov space $B^{\al,\f}_{p}$ with $\al>1/2$. We prove that the Besov solution satisfying the above mentioned condition is unique in the class of dissipative solutions. In the later part of this article, we prove that the one sided Lipschitz condition gives uniqueness among weak solutions with the Besov regularity, $B^{\al,\f}_{3}$ for $\al>1/3$. 
			 Our proof relies on commutator estimates for Besov functions and the relative entropy method.
		\end{abstract}
		
	\end{frontmatter}
\vspace*{-0.75cm}
	\tableofcontents
	\section{Introduction}
	In 1757, Euler formulated the motion of a compressible inviscid fluid by the following system
	\begin{eqnarray}
	\pa_t\rho+\mbox{div}_x(\rho\textbf{u})&=&0,\label{I1}\\
	\pa_t(\rho\textbf{u})+\mbox{div}_x(\rho\textbf{u}\otimes\textbf{u})+\nabla_xp(\rho,\vt)&=&0,\label{I2}\\
	\pa_t\left(\frac{1}{2}\rho|\textbf{u}|^2+\rho e(\rho,\vt)\right)+\mbox{div}_x\left[\left(\frac{1}{2}\rho|\textbf{u}|^2+\rho e(\rho,\vt)+p(\rho,\vt)\right)\textbf{u}\right]&=&0\label{I3},
	\end{eqnarray}
	where $\rho(t,x),\textbf{u}(t,x), \vt(t,x)$ represent respectively the density, velocity and temperature of the fluid. In the above system, the quantity $e(\rho,\vt)$ represents the internal energy of the fluid and the function $p$ denotes pressure. We suppose that the entropy of the fluid $s(\rho,\vt)$ is satisfying the following inequality	
	\begin{eqnarray}
	\pa_t(\rho s(\rho,\vt))+\mbox{div}_x(\rho s(\rho,\vt)\textbf{u})&\geq&0.\label{I4}
	\end{eqnarray}
	We also assume that entropy $s$, pressure $p$ and internal energy $e$ are linked by the following  \textit{Gibb's identity} 
	\begin{equation}\label{Gibbs}
	\vt Ds(\rho,\vt)=De(\rho,\vt)+p(\rho,\vt)D\left(\frac{1}{\rho}\right).
	\end{equation}
	For an ideal gas, the pressure $p$ and the internal energy $e$ are related by the following identity
	\begin{equation}\label{CES}
	p=\rho(\g-1)e,
	\end{equation}
	where the constant $\g$ is the adiabatic index determined as $\g=c_{P}/c_{V}>1$ and $c_{P},c_{V}$ are the specific heat capacities at constant pressure and volume respectively. The identity (\ref{CES}) is also called the caloric equation of states. In addition, we would like to state the Boyle's law:
	\begin{equation}\label{Boyle}
	p=\rho\vt.
	\end{equation}
    This yields $e=c_{V}\vt$ for $c_{V}(\g-1)=1$. 
    \par To avoid the kinematic boundary terms we assume that solutions are spatially periodic and we work with the following domain for space variable,
    \begin{equation}\label{domain}
    \Omega=\left([-1,1]_{\{\pm1\}}\right)^N.
    \end{equation} 
    We are mostly interested in the cases where $N=2,3$.
	\par It is interesting to note that the identities (\ref{Gibbs}), (\ref{CES}) and (\ref{Boyle}) give the following structure of the entropy $s$,
	\begin{equation*}
		s(\rho,\vt)=\log\left(\frac{\vt^{c_{V}}}{\rho}\right).
	 \end{equation*}
	 
	 \par Kato \cite{Kato} established the local existence of a smooth solution for the general hyperbolic system. It is well known that due to hyperbolicity, singularity can appear in finite time even for the smooth data (see for instance, \cite{Smoller}). That is why 
	  it is preferable to work with weak solution for the system (\ref{I1})--(\ref{I3}). We use the following weak formulation of the Euler system \eqref{I1}--\eqref{I3} corresponding to initial data $(\rho_0,\textbf{u}_0,\vt_0)$. 
	\begin{itemize}
		\item The equation of continuity can be weakly formulated as
		\begin{equation}\label{W1}
		\int\limits_{0}^{\tau}\int\limits_{\Omega}\left[\rho\pa_t\varphi+\rho\textbf{u}\cdot\nabla_x\varphi\right]dxdt=\int\limits_{\Omega}\rho(\tau,\cdot)\varphi(\tau,\cdot)dx-\int\limits_{\Omega}\rho_0(\cdot)\varphi(0,\cdot)dx
		\end{equation}
		for any $\varphi\in C_c^{\f}\left([0,T)\times\Omega,\re\right)$ and $0<\tau<T$.
		\item The weak formulation of the momentum equation is the following
		\begin{equation}\label{W2}
		\begin{array}{lll}
		\dip\int\limits_{0}^{\tau}\int\limits_{\Omega}\left[\rho\textbf{u}\cdot\pa_t\psi+\rho\textbf{u}\otimes\textbf{u}:\nabla_x\psi+p(\rho,\vt)\mbox{div}_{x}\psi\right]dxdt&=\dip\int\limits_{\Omega}\rho(\tau,\cdot)\textbf{u}(\tau,\cdot)\cdot\psi(\tau,\cdot)dx\\
		&\dip-\int\limits_{\Omega}\rho_0(\cdot)\textbf{u}_0(\cdot)\cdot\psi(\tau,\cdot)dx
		\end{array}
		\end{equation}
		for any $\psi\in C_c^{\f}\left([0,T)\times\Omega,\re^N\right)$ and $0<\tau<T$.
		
		\item We say the  energy equation is satisfied weakly if the following holds
		\begin{equation}\label{W3}
		\int\limits_{0}^{\tau}\int\limits_{\Omega}\left[E\pa_t\varphi+\left(E+p(\rho,\vt)\right)\textbf{u}\cdot\nabla_x\varphi\right]dxdt=\int\limits_{\Omega}E(\tau,\cdot)\varphi(\tau,\cdot)dx-\int\limits_{\Omega}E(0,\cdot)\varphi(0,\cdot)dx
		\end{equation}
		for any $\varphi\in C_c^{\f}\left([0,T)\times\Omega,\re\right)$ and $0<\tau<T$ where $E(t,x):=\frac{1}{2}\rho|\textbf{u}|^2+\rho e(\rho,\vt)$.
		\item We say the solution is \textit{admissible} if it satisfies the weak formulation (\ref{W1})--(\ref{W3}) along with the following weak formulation of entropy inequality
		\begin{equation}\label{W4}
		\begin{array}{lll}
		\dip\int\limits_{0}^{\tau}\int\limits_{\Omega}\left[\rho \Psi(s(\rho,\vt))\pa_t\varphi+\rho \dip\Psi(s(\rho,\vt))\textbf{u}\cdot\nabla_x\varphi\right]dxdt
		&\leq&\dip\int\limits_{\Omega}\rho(\tau,\cdot)\Psi(s(\rho(\tau,\cdot),\vt(\tau,\cdot)))\varphi(\tau,\cdot)dx\\
		&-&\dip\int\limits_{\Omega}\rho_0(\cdot)\Psi(s(\rho_0(\cdot),\vt_0(\cdot)))\varphi(0,\cdot)dx
		\end{array}
		\end{equation}
		for any $0\leq\varphi\in C_c^{\f}\left([0,T)\times\Omega,\re\right),\Psi\in BC(\re),\Psi^{\p}\geq0$  and $0<\tau<T$. The function space $BC(\re)$ denotes the space of bounded continuous functions on $\re$.
	\end{itemize}  
	It is important to mention that the well--posedness of the Euler system (\ref{I1})--(\ref{I3}) is still largely unavailable. In his seminal work \cite{Diperna_measure-valued}, DiPerna gave the concept of measure valued solution to gas dynamics. That is to describe the solution in terms of a sequence of Young measures. B\v{r}ezina and Feireisl \cite{Brezina} introduced the notion of \textit{dissipative measure valued solution} to the complete Euler system. A detailed definition can be found in subsection \ref{defn}. Note that weak solutions are included in this category.
	\par For compressible isentropic Euler system the ill--posedness results exist due to Chiodaroli et al. \cite{Chio}. They used the convex integration technique which was introduced in fluid mechanics by De Lellis and  Sz\'ekelyhidi \cite{DLS09,Delellis}.  Feireisl et al. \cite{FKKM} showed that there exist infinitely many admissible weak solutions to the complete Euler system. In all of these results, construction of infinitely many solutions holds due to the presence of shock in solution.
	 \par Besides the above mentioned non-uniqueness results, there are few works on uniqueness of solutions to system \eqref{I1}--\eqref{I3} as well. In one dimension, DiPerna \cite{Diperna_uniqueness} proved the stability of piecewise Lipschitz solutions of strictly hyperbolic system within the class of entropy solutions. Chen and Frid \cite{ChenFrid} proved the uniqueness of 1-D rarefaction solution of the complete Euler system \eqref{I1}--\eqref{I3}. The stability of 1-D Riemann problem solution to complete Euler system has been studied in \cite{ChenFridLi}. Feireisl and Kreml \cite{FeiKre} showed the uniqueness of rarefaction waves for compressible Euler system of barotropic flow in multi dimension. For the isentropic Euler system, the uniqueness of Besov solution satisfying a one-sided Lipschitz bound has been proved \cite{FGJ} by using commutator estimate. It was the first time when commutator estimate has been used in the context of weak-strong uniqueness results. The stability of rarefaction solutions of one dimensional Riemann problem in multi-D has been shown \cite{FeiVas} for the system (\ref{I1})--(\ref{I3}). B\v{r}ezina and Feireisl \cite{Brezina} proved a weak(measure valued)-strong uniqueness for full Euler system by using a relative entropy previously introduced in \cite{FeiNo} for Navier-Stokes-Fourier system. Dafermos \cite{Daf73} proposed a principle of \textit{maximal dissipation} as a uniqueness criterion for physical solution. B\v{r}ezina and Feireisl \cite{Brezina2} obtained the class of maximal solution  subject to entropy production rate for full Euler system. Therefore, the uniqueness problem for complete Euler system remains unanswered for a large class of solutions.  
	 Relative entropy method has been applied to other systems also to prove weak-strong uniqueness (see for instance,  \cite{DeDoGwGwWi} and \cite{GwGwWi}). For a good survey on weak-strong uniqueness in fluid mechanics we refer interested reader to \cite{Wiedemann}. 
	
	\par In this article, we consider \textit{dissipative solution} (see Definition \ref{Dissipative}) and a weak solution having Besov regularity $B^{\al,\f}_{q}$ with $\al>1/2$. Then, we show that these two solutions coincide if the Besov solution satisfies a one-sided Lipschitz condition. We assume the Besov regularity only for the positive time and this allows shock-free Riemann problem solutions to be included in this class of Besov solution with one-sided Lipschitz condition (\ref{Lip}). We observe that one-dimensional Riemann problem solutions in multi-D does not have uniform Lipschitz bound up to time $t=0$ but they do satisfy one-sided Lipschitz condition \eqref{Lip} with appropriate Besov regularity. Our result (see Theorem \ref{theorem1}) improves the weak(measured valued)-strong uniqueness result of B\v{r}ezina-Feireisl \cite{Brezina}. 
	In the later part of this article, we obtain a uniqueness result for Besov regularity $\al>1/3$. We prove that a weak solution satisfying one sided Lipschitz bound is unique in the class of  $B^{\al,\f}_{3}, \al>1/3$ regular weak solutions of the complete Euler system \eqref{I1}--\eqref{I3}. We also obtain a similar uniqueness result for isentropic Euler system. This is different from the uniqueness result of \cite{FGJ} as we focus on establishing stability among weak solutions with regularity $B^{\al,\f}_3, \al>1/3$.
	We use a relative entropy \cite{Brezina} suitable for complete Euler system \eqref{I1}--\eqref{I3}. Since we do not have the pointwise sense of the system (\ref{I1})--(\ref{I3}) for the Besov solution, we work with mollified version of the solution and use it as a test function. To prove the uniqueness result among $B^{\al,\f}_3,\,\al>1/3$ we need to mollify both weak solutions. Then we pass to the limit by using commutator estimates. Similar type of commutator estimate lemmas can be found in \cite{CET,FGJ,GwMiGw}.
	\par One key step of our proof relies on commutator estimate for Besov functions. In fluid mechanics, it has been previously used mostly for proving energy equality. Constantin et al. \cite{CET} proved positive result of Onsager's  conjecture by using a commutator estimate for Besov functions. Later Feireisl et al. \cite{FeGwGwWi} gave one sufficient condition on Besov regularity of solution to get the energy conservation for compressible isentropic Euler system.  Analogous results are proven for complete Euler system \cite{Drivas} and also for general conservation laws \cite{GwMiGw}. As we have mentioned earlier commutator estimate has been used in \cite{FGJ} to show uniqueness of dissipative solutions to the isentropic compressible Euler system. 
	
	
	\par Rest of the article is organized as follows. In subsection \ref{outline} we give the outline of our method. In the subsection \ref{defn} we give the complete definition of \textit{dissipative measure valued solution} and \textit{dissipative solution}. After that in subsection \ref{main-result} we state our main results and prove them in section \ref{proof1} and \ref{proof2}. Section \ref{sec:relative} is devoted for proving a relative entropy inequality for the system \eqref{I1}--\eqref{I4}.
	\subsection{Outline of the method}\label{outline}
	\begin{itemize}
		\item Our main vehicles of the proof are relative entropy and the commutator estimate. It can be seen as variant of weak-strong uniqueness. The main difference between weak-strong uniqueness and our method is that we have considered measure valued solution in place of weak solution and instead of strong solution we have considered a weak solution with certain Besov regularity.
		
		\item In the context of weak-strong uniqueness, natural choice of the test function is the strong solution. In our situation we do not have that flexibility since the regular solution is not $C^1$ rather Besov. To avoid this restriction, we mollify them and use the mollified version as test function.
		
		\item By using a suitable commutator estimate, we pass to the limit in mollifiers sequence to get back the actual Besov solution we have started with. Rest of the proof follows from Gr\"onwall's lemma.
	\end{itemize}
	\subsection{Dissipative measure-valued solutions}\label{defn}
	In this section, we are going to recall the notion of \textit{dissipative solutions} of the system \eqref{I1}--\eqref{I4} from \cite{BFH}. Similar notion of dissipative solutions in the context of isentropic compressible Euler system can be found in \cite{BFH1}.
	\par Consider the following subset of $\re^{N+2}$, $	\tQ:=\left\{[ \td,\tm,\tS ]\left| \td\geq0,\tm\in\re^N,\tS\in\re\right.\right\}$.
	Let $\tP$ be the set of all probability measures on $\tQ$. By $\me$ and $\tme$ we denote the set of positive Radon measures on $\bar{\Omega}$ and $S^{N-1}\times\bar{\Omega}$ respectively. A \textit{dissipative measure valued solution} of Euler equations (\ref{I1})--(\ref{I4}) corresponding to the initial data $[ \td_0,\tm_0,\tS_0]$  and total energy $E_0$ 
	 is consisting of the following three things
	\begin{enumerate}[(i)]
		\item a family of probability measures parametrized by $(t,x)$
		\begin{equation*}
		\tY_{t,x}:(0,\f)\times\Omega\rr\tP,\ \tY\in L^\f_{\mbox{weak}-(*)}\left((0,T)\times\Omega;\tP\right),
		\end{equation*} 
		\item two type of energy concentration defect measures namely, internal and kinetic defect measures
		\begin{equation*}
		\tDkin,\tDint\in L^\f_{\mbox{weak}-(*)}\left((0,\f);\me\right),
		\end{equation*} 
		\item a convection concentration defect measure 
		\begin{equation*}
		\tDcon\in L^{\f}_{weak-(*)} \left((0,\f);\tme\right).
		\end{equation*}
	\end{enumerate}
	For $0\leq \si<\tau<T$ the above mentioned quantities verify Euler system (\ref{I1})--(\ref{I4}) in the following sense
	\begin{itemize}
		\item the continuity equation (\ref{I1}) or mass conservation
		\begin{equation}\label{D1}
		\int\limits_{\si}^{\tau}\int\limits_{\Omega}\left[\langle\tY_{t,x}; \td\rangle\pa_t\varphi+\langle\tY_{t,x};\tm\rangle\cdot\nabla_x\varphi\right]dxdt=\left[\int\limits_\Omega\lla \tY_{t,x}; \td\rra\varphi(t,\cdot)\,dx\right]_{t=\si}^{t=\tau}
		\end{equation}
				for any $\varphi\in C_c^{\f}\left([0,T)\times\Omega,\re\right)$.
		\item Momentum conservation 
		\begin{equation}\label{D2}
		\begin{array}{lll}
		& & \dip\dom\left[\left\langle\tY_{t,x};\tm\right\rangle\cdot\pa_t\psi+\left\langle\tY_{t,x};\frac{\tm\otimes\tm}{ \td}\right\rangle:\nabla_x\psi+\langle\tY_{t,x} ;p( \td,\tS)\rangle\mbox{div}_x\psi\right]dxdt\\
		&+&\dip\int\limits_{\si}^{\tau}\left[\int\limits_\Omega\int\limits_{S^{N-1}}((\xi\otimes\xi):\nabla_x\psi) d\tDcon(t)\right]dt+(\g-1)\int\limits_{\si}^{\tau}\left[\int\limits_\Omega(\mbox{div}_x\psi) d\tDint(t)\right]dt\\
		&=&\dip\left[\int\limits_{\Omega}\lla\tY_{t,x};\tm\rra\cdot\psi(t,\cdot)dx\right]_{t=\si}^{t=\tau}
		\end{array}
		\end{equation}
				for any $\psi\in C_c^{\f}\left([0,T)\times\Omega,\re^N\right)$.
		\item Energy conservation
		\begin{equation}\label{D3}
		\int\limits_{\Omega}\left\langle\tY_{t,x};\frac{1}{2}\frac{|\tm|^2}{ \td}+c_{V} (\td)^\g \mbox{exp}\left(\frac{\tS}{c_{V} \td}\right)\right\rangle dx+\int\limits_{\bar{\Omega}}\left(d\tDkin(t)+d\tDint(t)\right)=E_0
		\end{equation} 
		for a.e. $t\geq 0$ where $E_0:=\int\limits_{\Omega}\frac{1}{2}\frac{|\tm_0|^2}{ \td_0}+c_{V} (\td_0)^\g \mbox{exp}\left(\frac{\tS_0}{c_{V} \td_0}\right)dx$.
		\item Entropy inequality
		\begin{equation}\label{D4}
		\begin{array}{llll}
		& &\dip\dom\left[\lla\tY_{t,x}; \td \Psi\left(\frac{\tS}{ \td}\right)\rra\pa_t\varphi+\lla\tY_{t,x};\Psi\left(\frac{\tS}{ \td}\right)\tm\rra\cdot\nabla_x\varphi\right]dxdt\\
		&\leq&\dip\left[\int\limits_{\Omega}\lla\tY_{t,x}; \td\Psi\left(\frac{\tS}{ \td}\right)\rra\varphi(t,\cdot)dx\right]_{t=\si}^{t=\tau}
		\end{array}\end{equation}
				for any $0\leq\varphi\in C_c^{\f}\left([0,T)\times\Omega,\re\right),\Psi\in BC(\re),\Psi^{\p}\geq0$.
	\end{itemize}

	\begin{definition}[\cite{BFH}]\label{DMVS}
		A dissipative measure valued (DMV) solution of the system (\ref{I1})--(\ref{I4}) corresponding to the initial data $[ \td_0,\tm_0,\tS_0]$ and the initial energy $E_0$, is a parametrized family of probability measures
		\begin{equation*}
		\tY_{t,x}:(0,\f)\times\Omega\rr\tP,\ \tY\in L^\f_{\mbox{weak}-(*)}\left((0,T)\times\Omega;\tP\right),
		\end{equation*}
		along with the energy concentration defect measures 
		\begin{equation*}
		\tDkin,\tDint\in L^\f_{\mbox{weak}-(*)}\left((0,T)\times\Omega;\me\right)
		\end{equation*}
		and the convection concentration defect measure 
		\begin{equation*}
		\tDcon\in L^\f_{\mbox{weak}-(*)}\left((0,T)\times\Omega;\tme\right),\  \frac{1}{2}\int\limits_{S^{N-1}}d\tDcon=\tDkin
		\end{equation*}
		satisfying (\ref{D1}),(\ref{D2}),(\ref{D3}) and (\ref{D4}).
	\end{definition}
	\begin{definition}\label{Dissipative}
		The tuple $\left([\tds,\tms,\tSs],E_0\right)$ with
		\begin{equation}
		\begin{array}{lll}
		\tds\in C_{\mbox{weak,loc}}\left([0,\f); L^{\g}(\Omega)\right), \tms\in C_{\mbox{weak,loc}}\left( [0,\f);L^{\frac{2\g}{\g+1}}(\Omega,\re^N)\right),\\
		\tSs\in L^\f\left((0,\f); L^\g(\Omega)\right)\cap BV_{\mbox{weak,loc}}\left([0,\f);W^{-l,2}(\Omega)\right),l>\frac{N}{2}+1,
		\end{array}
		\end{equation}
		is called a \textit{dissipative solution} of the system (\ref{I1})--(\ref{I4}) with the initial data
		\begin{equation*}
		\left([\tds_0,\tms_0,\tSs_0],E_0\right)\in L^{\g}(\Omega)\times L^{\frac{2\g}{\g+1}}(\Omega,\re^N)\times L^{\g}(\Omega)\times[0,\f)
		\end{equation*}
		if there exists a dissipative measure valued (DMV) solution $\tY_{t,x}$ as in definition \ref{DMVS} such that the following holds
		\begin{equation}
		\tds(t,x)=\lla\tY_{t,x};\td\rra,\ \tms(t,x)=\lla\tY_{t,x};\tm\rra,\ \tSs(t,x)=\lla\tY_{t,x};\tS\rra.
		\end{equation}
	\end{definition}
  
	\subsection{Main result}\label{main-result}

	\begin{theorem}\label{theorem1}
	 Let $p,e,s$ be related by (\ref{Gibbs}), (\ref{CES}) and (\ref{Boyle}). Let $[\rho,\tms,\tSs]$ be a dissipative solution to the complete Euler system (\ref{I1})--(\ref{I4}) according to Definition \ref{Dissipative} with the initial data $[\rho_0,\tms_0,\tSs_0]$. Let the trio $[\tr,\mathfrak{u},\tT]$ be a weak solution to the system (\ref{I1})--(\ref{I3}) with the initial data $[\rho_0,\mathfrak{u}_0,\tT_0]$ such that $\tms_0=\rho_0\mathfrak{u}_0$ and $\tSs_0=\rho_0s(\rho_0,\tT_0)$ hold. Moreover, we assume that the weak solution $[\tr,\mathfrak{u},\tT]$ satisfying the following properties:
		\begin{enumerate}[(a)]
			\item\label{a} There exist $r,R,C>0$ such that
			\begin{equation*}
			0<r<\tr<R, |\mathfrak{u}|\leq C \mbox{ for a.e. }(t,x)\in(0,T)\times\Omega.
			\end{equation*}
			\item \label{b}  The temperature $\tT\in W^{1,\f}((\de,T)\times\Omega)\cap C\left([0,T],L^1(\Omega,\re)\right)$. The density $\mathfrak{r}\in C\left([0,T],L^1(\Omega)\right)$ and the velocity $\mathfrak{u}=(\mathfrak{u}^1,\cdots,\mathfrak{u}^N)\in C\left([0,T],L^1(\Omega,\re^N)\right)$ satisfy the following regularity assumption  
			\begin{equation}\label{alpha-beta}
			\begin{array}{lll}
			\tr\in B^{\al,\f}_q((\de,T)\times\Omega)
			 \mbox{ and }\mathfrak{u}\in B^{\B,\f}_q((\de,T)\times\Omega,\R^N)\mbox{ for each }\de>0
			\end{array}
			\end{equation}
			with $\B>\max\{\al,1-\al,1/2\},\ q\geq\frac{4\g}{\g-1}$.
			\item\label{c} 
			There exists $\mathcal{C}\in L^1((0,T),[0,\f))$ such that velocity $\mathfrak{u}$ satisfies the following inequality
			\begin{equation} \label{Lip}
			\int\limits_{\Omega}{ \left[ - \xi \cdot \mathfrak{u}(\tau, \cdot) (\xi \cdot \nabla_x) \varphi  + \mathcal{C}(\tau) |\xi|^2 \varphi \right] } \geq 0 \mbox{ for all }\xi \in \re^N\mbox{ and }\varphi \in C^\f(\Omega),\varphi \geq 0.
			\end{equation}
		\end{enumerate}
		Then the following holds	
		\begin{equation*}
		\rho\equiv\tr,\ \tms\equiv\tr\mathfrak{u}\mbox{ and }\tSs\equiv \mathfrak{r}s(\mathfrak{r},\tT)\mbox{ for a.e. }(t,x)\in(0,T)\times\Omega.
		\end{equation*} 
		
	\end{theorem}
	 Since dissipative solutions are more general concept and the weak solutions are already included in this category the similar result is true for admissible solutions also.
	\begin{corollary}\label{corollary1}
		Let $[\rho,\tms,\Theta]$ be an admissible  solution to the complete Euler system (\ref{I1})--(\ref{I4}) in the sense of (\ref{W1})--(\ref{W4}) with initial data $[\rho_0,\tms_0,\Theta_0]$. Let the trio $[\tr,\mathfrak{u},\tT]$ be a weak solution to the system (\ref{I1})--(\ref{I3}) with the initial data $[\rho_0,\mathfrak{u}_0,\Theta_0]$ such that $\tms_0=\rho_0\mathfrak{u}_0$. Assume that $[\tr,\mathfrak{u},\tT]$ satisfies the assumptions \ref{a},\ref{b} and \ref{c} as in Theorem \ref{theorem1}.	Then we have
		\begin{equation*}
		\rho\equiv\tr,\ \tms\equiv\tr\mathfrak{u}\mbox{ and }\Theta\equiv\tT\mbox{ for a.e. }(t,x)\in(0,T)\times\Omega.
		\end{equation*} 
	\end{corollary}

Suppose two weak solutions coming from same initial condition satisfy regularity assumption \ref{b}. Then it is clear from Corollary \ref{corollary1} that they must coincide if one of them satisfies the inequality \eqref{Lip}. Therefore, the condition \eqref{Lip} establishes uniqueness among the class of weak solution verifying regularity \ref{b}. In the next theorem, we prove that the one sided Lipschitz condition gives uniqueness in a larger space than \eqref{alpha-beta}.
	
	\begin{theorem}\label{theorem2}
	Let $p,e,s$ be related by (\ref{Gibbs}), (\ref{CES}) and (\ref{Boyle}). Let $(\rho,\textbf{u},\vartheta)$ and $(r,\textbf{v},\theta)$ be two weak solutions to the complete Euler system (\ref{I1})--(\ref{I3}) with the initial data $(\rho_0,\textbf{v}_0,\vartheta_0)$. Suppose $(\rho,\textbf{u},\vartheta)$ satisfies \eqref{I4} in the sense of distribution. Moreover, we assume the following properties:
	\begin{enumerate}[(A)]
		\item \label{a1}  Temperatures $\vartheta,\theta\in W^{1,\f}((\de,T)\times\Omega)\cap C\left([0,T],L^1(\Omega,\re)\right)$. Densities ${r},\rho\in C\left([0,T],L^1(\Omega)\right)$ and velocities $\textbf{u},\textbf{v}\in C\left([0,T],L^1(\Omega,\re^N)\right)$ satisfy the following regularity assumption  
		\begin{equation}\label{alpha-beta1}
		(\rho,\textbf{u}),(r,\textbf{v})\in B^{\B,\f}_3((\de,T)\times\Omega,\R^{1+N})\mbox{ for each }\de>0\mbox{ with }\B>{1}/{3}.
		\end{equation}

    	\item\label{b1} Assumption \ref{a} in Theorem \ref{theorem1} holds for $r$ and $\textbf{v}$ with some $\underline{r},\overline{r},C>0$. We also assume that $\textbf{v}$ satisfies \eqref{Lip} for some $\mathcal{C}_1\in L^1((0,T),[0,\f))$.  
	\end{enumerate}
	Then the following holds,	
	\begin{equation*}
	\rho\equiv r,\ \textbf{u}\equiv\textbf{v}\mbox{ and }\vartheta\equiv \theta\mbox{ for a.e. }(t,x)\in(0,T)\times\Omega.
	\end{equation*} 
	
\end{theorem}	

	 \begin{remark}
It is important to note that the solutions consisting one dimensional rarefaction waves are included in class of Besov solutions (\ref{alpha-beta}) with one sided Lipschitz condition (\ref{Lip}). We can consider a one dimensional Riemann data $u_0(x_1)$ for system (\ref{I1})--(\ref{I3}) such that only rarefaction appears in the solution. Then we extend that data in multi dimension by keeping it invariant in other variables. For more on Riemann problem solutions one can check \cite{CH}.
\end{remark}
    \begin{remark}
     Working on the domain as defined in (\ref{domain}) is really not a restriction. One can also work with non-periodic data and domain as well. Then it can be reduced to our situation by following the similar analysis as in  \cite[section 6]{FGJ}.
    \end{remark}

	\subsection{Preliminaries and notations}
	Suppose $\B\in(0,1)$, $p\geq1$ and $\mathcal{U}\subset\bar{\mathcal{U}}\subset(0,T)\times\Omega$ the Besov semi-norm is defined as follows
	\begin{equation*}
	|g|_{B^{\B,\f}_{p}(\mathcal{U})}=\sup\limits_{0\neq h\in\re^{N+1},\ \mathcal{U}+h\subset(0,T)\times\Omega}\frac{\|g(\cdot+h)-g(\cdot)\|_{L^p(\mathcal{U})}}{|h|^{\B}}.
	\end{equation*}
	For more on Besov space we refer interested reader to \cite{Triebel}. Let $\{\eta_{\e}\}_{\e>0}$ be standard mollifier sequence. Then for a function $g\in B^{\B,\f}_{p}$ we have the following estimates for $\e>0$
	\begin{align}
	\|g*\eta_{\e}-g\|_{L^{p}(\mathcal{U})}&\leq |g|_{B^{\B,\f}_{p}(\mathcal{U})}\e^{\B},\label{Besov_estimate1}\\
	\|g(\cdot+h)-g(\cdot)\|_{L^{p}(\mathcal{U})}&\leq|g|_{B^{\B,\f}_{p}(\mathcal{U})} |h|^{\B},\label{Besov_estimate2}\\
		\|\nabla(g*\eta_{\e})\|_{L^{p}(\mathcal{U})}&\leq |g|_{B^{\B,\f}_{p}(\mathcal{U})}\e^{\B-1}.\label{Besov_estimate3}
	\end{align}
	Above estimates follow from the definition and careful analysis.
	\section{Relative entropy}\label{sec:relative}
	Notion of relative entropy comes from seminal work of DiPerna \cite{Diperna_uniqueness} and Dafermor \cite{Daf79}. Here we use  the relative energy in Eulerian coordinate introduced in \cite{FeiNo}. Before we define relative energy let us first recall definition of some useful quantities. We define \textit{ballistic free energy} as follows
	\begin{equation}\label{ballistic}
	H_{\tT}(\tr,\theta):=\tr e(\tr,\theta)-\tT\tr s(\tr,\theta).
	\end{equation}
	By using Gibbs identity and \eqref{ballistic} we observe that
	\begin{equation}\label{P2}
	\begin{array}{ll}
	&\tr \pa_{\tr}H_{\tT}(\tr,\tT)=H_{\tT}(\tr,\tT)+p(\tr,\tT),\,\tr\pa_{\tr}s(\tr,\tT)=-\frac{1}{\tr}\pa_{\tT}p(\tr,\tT)\\
	\mbox{ and } &\pa_{\tT}H_{\tT}(\tr,\tT)=-\tr s(\tr,\tT).
	\end{array}
	\end{equation}
%
	For technical reason we would like to consider the following quantity which is comparable to the temperature $\vt$.
	\begin{equation}\label{theta}
	{\Theta}:=(\tds)^{\g-1}\mbox{exp}\left\{(\g-1)\frac{\tSs}{\tds}\right\} \mbox{	for $\tds\in(0,\f)$ and $\tSs\in\re$.}
	\end{equation}
	This implies $\tSs=\tds s(\tds,{\Theta})$. 
	
	Now we are ready to define the \textit{relative entropy} 
	\begin{equation}\label{rel-entropy}
	\begin{array}{lll}
	\mathfrak{E}\left(\tds,\tms,\Theta\left|\right.\tr,\mathfrak{u},\tT\right)&:=
	\frac{1}{2}\tds\left|\frac{\tms}{\rho}-\mathfrak{u}\right|^2
	+\rho e(\rho,\Theta)\\
	&-\rho\tT\Psi\left(s(\rho,\Theta)\right)-\pa_{\tr}H_{\tT}(\tr,\tT)\left(\rho-\tr\right)-H_{\tT}(\tr,\tT).
	\end{array}
	\end{equation}
	Less formally, $\mathfrak{E}\left(\tds,\tms,\Theta\left|\right.\tr,\mathfrak{u},\tT\right)$ measures a distance between vectors $(\tds,\tms,\Theta)$ and $(\tr,\mathfrak{u},\tT)$. For relative energy in Lagrangian form see \cite{ChenFrid}.
	\vspace*{-0.25cm}
	\subsection{Relative entropy inequality}
	If $(\tds,\tms,\Theta)$ is a dissipative solution, then $\int\limits_{\Omega}\mathfrak{E}\left(\tds,\tms,\Theta\left|\right.\tr,\mathfrak{u},\tT\right)$ satisfies a certain estimate which plays a key role in our proof of Theorem \ref{theorem1}. More precisely, we have the following,

	\begin{proposition}\label{proposition1}
	Let $[\tds,\tms,\tSs]$ be a dissipative solution defined as in Definition \ref{Dissipative}. Let
		\begin{equation*}
	\mathfrak{r}\in C^1\left([0,T]\times\Omega\right), \mathfrak{r}>0, \mathfrak{T}\in C^1\left([0,T]\times\Omega\right), \mathfrak{u}\in C^1\left([0,T]\times\Omega,\re^N\right).
	\end{equation*}
	Then, for $0\leq \si<\tau\leq T$ we have
		\begin{align}
		& \dip\left[\int\limits_{\Omega}\left(\lla\tY_{t,x};\mathfrak{E}\left(\td,\tm,\bar{\Theta}\left|\right.
		\tr,\mathfrak{u},\tT\right)\rra +d\tDkin(t)+d\tDint(t)\right)dx\right]_{t=\si}^{t=\tau}\nonumber\\
		&\leq
		\domm\lla\tY_{t,x};\frac{\left(\tm-\td\mathfrak{u}\right)\otimes\left(\td\mathfrak{u}-\tm\right)}{\td}\rra:\nabla_x\mathfrak{u}\,dxdt\nonumber\\
		&-\domm\frac{\pa_{\tT}H_\tT}{\tr}
		\left(\lla\tY_{t,x}; \td\rra\pa_t\tT+\langle\tY_{t,x};\tm\rangle\cdot\nabla_x\tT\right)\,dxdt\nonumber\\
		&-\domm\lla\tY_{t,x}; \td \Psi\left(s(\td,\bar{\Theta})\right)\rra\pa_t\tT+\lla\tY_{t,x};\Psi\left(s(\td,\bar{\Theta})\right)\tm\rra\cdot\nabla_x\tT
		\,dxdt\nonumber\\
		&+\int\limits_{\si}^{\tau}\int\limits_{\Omega}\left[
		\left(1-\lla\tY_{t,x};\frac{\td}{\tr}\rra\right)\pa_t\left(p(\tr,\tT)\right)
		-\lla\tY_{t,x};\frac{\td\mathfrak{u}}{\tr}\rra\cdot\nabla_x\left(p(\tr,\tT)\right)\right]\,dxdt\nonumber\\
		&+\int\limits_{\si}^{\tau}\int\limits_{\Omega}\frac{1}{\tr}\left[
		\lla\tY_{t,x};\td\mathfrak{u}-\tm\rra\cdot\left(\tr\pa_t\mathfrak{u}+\tr\nabla_x\mathfrak{u}\cdot\mathfrak{u}+\nabla_x p(\tr,\tT)\right)\right]\,dxdt\nonumber\\
		&-\domm\lla\tY_{t,x} ;p( \td,\bar{\Theta})\rra\mbox{div}_x\mathfrak{u}\,dxdt
		-\mathcal{R}(\si,\tau),\label{rel-en-es}
		\end{align}
		where 
		\begin{equation}\label{R1}
		\mathcal{R}(\si,\tau):=\dip\int\limits_{\si}^{\tau}\left[\int\limits_\Omega\int\limits_{S^{N-1}}((\xi\otimes\xi):\nabla_x\mathfrak{u}) d\tDcon(t)\right]dt+(\g-1)\int\limits_{\si}^{\tau}\left[\int\limits_\Omega (\mbox{div}_x\mathfrak{u}) d\tDint(t)\right]dt.
		\end{equation}
	\end{proposition}
\begin{proof}
		We split the proof into following steps.
		\begin{enumerate}
			\item Set $\varphi=\frac{1}{2}|\mathfrak{u}|^2$ in continuity equation (\ref{D1}) and get
			\begin{equation}\label{P11}
			\left[\int\limits_\Omega\lla \tY_{t,x}; \td\rra\frac{1}{2}|\mathfrak{u}|^2dx\right]_{t=\si}^{t=\tau}=
			\int\limits_{\si}^{\tau}\int\limits_{\Omega}\left[\langle\tY_{t,x}; \td\rangle\pa_t\mathfrak{u}\cdot\mathfrak{u}+\langle\tY_{t,x};\tm\rangle\cdot\nabla_x\mathfrak{u}\cdot\mathfrak{u}\right]dxdt.	
			\end{equation}
			\item Put $\psi=\mathfrak{u}$ in momentum equation (\ref{D2}) and obtain the following
			\begin{align}
			& \dip\left[\int\limits_{\Omega}\lla\tY_{t,x};\tm\rra\cdot\mathfrak{u}(t,\cdot)\,dx\right]_{t=\si}^{t=\tau}\nonumber\\
			&= \dip\domm\left[\lla\tY_{t,x};\tm\rra\cdot\pa_t\mathfrak{u}+\lla\tY_{t,x};\frac{\tm\otimes\tm}{ \td}\rra:\nabla_x\mathfrak{u}+\lla\tY_{t,x} ;p( \td,\bar{\Theta})\rra\mbox{div}_x\mathfrak{u}\right]\,dxdt\nonumber\\
			&+\dip\int\limits_{\si}^{\tau}\left[\int\limits_\Omega\int\limits_{S^{N-1}}((\xi\otimes\xi):\nabla_x\mathfrak{u}) \,d\tDcon(t)\right]dt+(\g-1)\int\limits_{\si}^{\tau}\left[\int\limits_\Omega (\mbox{div}_x\mathfrak{u}) \,d\tDint(t)\right]\,dt.\label{P12}
			\end{align}
			\item Take $\varphi=\tT$ in entropy inequality  (\ref{D4}) we have
			
			\begin{equation}\label{P13}
			\begin{array}{llll}
			&-&\dip\left[\int\limits_{\Omega}\lla\tY_{t,x}; \td\Psi\left(s(\td,\bar{\Theta})\right)\rra\tT(t,\cdot)dx\right]_{t=\si}^{t=\tau}\\
			&\leq &-\dip\domm\left[\lla\tY_{t,x}; \td \Psi\left(s(\td,\bar{\Theta})\right)\rra\pa_t\tT+\lla\tY_{t,x};\Psi\left(s(\td,\bar{\Theta})\right)\tm\rra\cdot\nabla_x\tT\right]dxdt.\\
			\end{array}\end{equation}
			\item Next we set $\varphi=\pa_{\tr}H(\tr,\tT)$ in continuity equation we get
			\begin{equation}\label{P14}
			\begin{array}{lll}
			& &\dip\left[\int\limits_\Omega\lla \tY_{t,x}; \td\rra\pa_{\tr}H(\tr,\tT)dx\right]_{t=\si}^{t=\tau}\\
			&=&	\dip\int\limits_{\si}^{\tau}\int\limits_{\Omega}\left[\langle\tY_{t,x}; \dip\td\rangle\pa_t\left(\pa_{\tr}H(\tr,\tT)\right)+\langle\tY_{t,x};\tm\rangle\cdot\nabla_x\left(\pa_{\tr}H(\tr,\tT)\right)\right]dxdt.
			\end{array}	
			\end{equation}
			\item From the definition of ballistic free energy $H_\tT$ we have
			\begin{align}
			\dip	\left[\int\limits_{\Omega}\left[\tr\pa_\tr H_ \tT(\tr, \tT)-H_{ \tT}(\tr, \tT)\right]dx\right]_{t=\si}^{t=\tau}&=\dip \left[\int\limits_{\Omega}p(\tr, \tT)dx\right]_{t=\si}^{t=\tau}
			=\dip\int\limits_{\si}^{\tau}\int\limits_{\Omega}\pa_t p(\tr, \tT)dxdt.\label{P15}
			\end{align}
		\end{enumerate}
		From \eqref{P2} we have 
		\begin{align}
		\mathfrak{r} \pa_{\mathfrak{r}\mathfrak{r}}H_{\mathfrak{T}}(\mathfrak{r},\mathfrak{T})&=\pa_\mathfrak{r} p(\mathfrak{r},\mathfrak{T}),\nonumber\\
		\mathfrak{r} \pa_{\mathfrak{r}\mathfrak{T}}H_{\mathfrak{T}}(\mathfrak{r},\mathfrak{T})&= \pa_{\mathfrak{T}}p(\mathfrak{r},\mathfrak{T})+\pa_\mathfrak{T} H_{\mathfrak{T}}(\mathfrak{r},\mathfrak{T}). 
		\end{align}
		Keeping this in mind we club \eqref{P11}--(\ref{P15}) with energy conservation (\ref{D3}) and obtain the following inequality
		\begin{equation*}
		\begin{array}{llll}
		& &\dip\left[\int\limits_{\Omega}\lla\tY_{t,x};\mathfrak{E}\left(\td,\tm,\bar{\Theta}\left|\right. \tr,\mathfrak{u},\tT\right)\rra
		dx+\int\limits_{\bar{\Omega}}\left(d\tDkin(t)+d\tDint(t)\right)dx\right]_{t=\si}^{t=\tau}\\
		&\leq& -\dip\domm\lla\tY_{t,x} ;p( \td,\bar{\Theta})\rra(\mbox{div}_x\mathfrak{u}) dxdt
		-\dip\int\limits_{\si}^{\tau}\int\limits_{\Omega}\frac{\pa_{\tT}H_{\mathfrak{T}}}{\tr}
		\left(\lla\tY_{t,x}; \td\rra\pa_t\tT+\langle\tY_{t,x};\tm\rangle\cdot\nabla_x\tT\right)dxdt\\
		&-&\dip\domm\left[\lla\tY_{t,x}; \td \Psi\left(s(\td,\bar{\Theta})\right)\rra\pa_t\tT+\lla\tY_{t,x};\Psi\left(s(\td,\bar{\Theta})\right)\tm\rra\cdot\nabla_x\tT\right]dxdt\\
		&+ &\dip\int\limits_{\si}^{\tau}\int\limits_{\Omega}\left[
		\left(1-\lla\tY_{t,x};\frac{\td}{\tr}\rra\right)\pa_t\left(p(\tr,\tT)\right)
		-\lla\tY_{t,x};\frac{\tm}{\tr}\rra\cdot\nabla_x\left(p(\tr,\tT)\right)\right]dxdt\\
		&+&	\dip\int\limits_{\si}^{\tau}\int\limits_{\Omega}
		\lla\tY_{t,x};\td\mathfrak{u}-\tm\rra\cdot\left[\pa_t\mathfrak{u}+\nabla_x\mathfrak{u}\cdot\mathfrak{u}\right]
		+\domm\left[\lla\tY_{t,x};\frac{\left(\tm-\td\mathfrak{u}\right)\otimes\left(\td\mathfrak{u}-\tm\right)}{\td}\rra:\nabla_x\mathfrak{u}\right]dxdt
		\\
		&-&\dip\int\limits_{\si}^{\tau}\left[\int\limits_\Omega\int\limits_{S^{N-1}}((\xi\otimes\xi):\nabla_x\mathfrak{u}) d\tDcon(t)\right]dt-(\g-1)\int\limits_{\si}^{\tau}\left[\int\limits_\Omega (\mbox{div}_x\mathfrak{u}) d\tDint(t)\right]dt.
        \end{array}
		\end{equation*}
		With further modification we show the inequality (\ref{rel-en-es}) with (\ref{R1}). This completes the proof of Proposition \ref{proposition1}.
	\end{proof}
%
%
%
%
%
%
%
%
%
%

 Now we want to recall some properties of the relative entropy. Detail analysis of these properties can be found in \cite{FeiNo2} and \cite{FeiNo} where it has been done in the context of full Navier-Stokes-Fourier system. 
 Let $r,R>$ and $\underline{\theta},\bar{\theta}>0$ such that the following holds
\begin{eqnarray}
\mathfrak{r}(t,x)\in[r,R]\  \mbox{ and }\ \mathfrak{\tT}(t,x)\in[\underline{\theta},\bar{\theta}]\ \ \mbox{ for } (t,x)\in(0,T)\times\Omega.
\end{eqnarray} 
Then we have the following
\begin{equation}\label{coercive}
\mathfrak{E}\left(\tds,\tms,\Theta\left|\right.\tr,\mathfrak{u},\tT\right)\geq C\left\{
\begin{array}{lll}
\dip\frac{1}{2}\rho\left|\frac{\tms}{\tds}-\mathfrak{u}\right|^2+|\tds-\mathfrak{r}|^2+\left|\Theta-\tT\right|^2&\mbox{if }
(\rho,\Theta)\in[r,R]\times [\underline{\theta},\bar{\theta}],&\\
\dip\frac{1}{2}\rho\left|\frac{\tms}{\tds}-\mathfrak{u}\right|^2+1+|\rho s(\rho,\Theta)|+ e(\rho,\Theta) &\mbox{otherwise,}&
\end{array}\right.
\end{equation}
the above constant $C$ depends on $r,R,\underline{\theta},\bar{\theta}$.


 Before we enter into the proof of our main theorem we would like to make few remarks and convenient assumptions on the measure valued solution. We give a proof for the measure valued solutions with $\td\in[r,R]$ and $\bar{\Theta}\in[\underline{\theta},\bar{\theta}]$ and with this assumption only one case arises in the (\ref{coercive}). The general case can be done by separating both the possibilities with a suitable cut-off function (see \cite{FeiNo}). However, we feel that the main focus of the current article is to understand the role of commutator estimate to get  uniqueness for complete Euler system with  minimal assumption on the weak solution. For the similar reason we take $\Psi(s)=s$ in (\ref{rel-en-es}) and the general case can be treated by a similar method as in \cite{Brezina}.

\section{Proof of Theorem \ref{theorem1}}\label{proof1}
Let $(\tr,\mathfrak{u},\mathfrak{T})$ be weak solution satisfying assumptions \ref{a}--\ref{c} as in Theorem \ref{theorem1}. We mollify $(\tr,\mathfrak{u},\mathfrak{T})$ by Friedrich mollifier $\eta_\e$ with $\e>0$ to obtain
	\begin{equation}
	\mr:=\tr*\eta_\e,\ \mv:=\mathfrak{u}*\eta_\e\mbox{ and } \mT:=\tT*\eta_\e.
	\end{equation}
 Now we apply Proposition \ref{proposition1} with $(\mr,\mv,\mT)$ to get
	\begin{align*}
		& \dip\left[\int\limits_{\Omega}\left(\lla\tY_{t,x};\mathfrak{E}\left(\td,\tm,\bar{\Theta}\left|\right.
		\mr,\mv,\mT\right)\rra+d\tDkin(t)+d\tDint(t)\right) dx\right]_{t=\si}^{t=\tau}\\
		&\leq
		\mathcal{J}^\e_1+\mathcal{J}^\e_2+\mathcal{J}^\e_3-\mathcal{R}(\si,\tau),
	\end{align*}
	where $\mathcal{J}^\e_1,\mathcal{J}^\e_2$ are defined as follows
	\begin{align}
		\mathcal{J}^\e_1 &:=	\domm\left[\lla\tY_{t,x};\frac{\left(\tm-\td\mv\right)\otimes\left(\td\mv-\tm\right)}{\td}\rra:\nabla_x\mv\right]dxdt,\label{def:J1}\\
		\mathcal{J}^\e_2 &:=-\dip\domm\left[\lla\tY_{t,x}; \td s(\td,\bar{\Theta})\rra\pa_t\mT+\lla\tY_{t,x};s(\td,\bar{\Theta})\tm\rra\cdot\nabla_x\mT\right]dxdt\nonumber\\
		&+ \dip\int\limits_{\si}^{\tau}\int\limits_{\Omega}\left[
		\lla\tY_{t,x};1-\frac{\td}{\mr}\rra\pa_t\left(p(\mr,\mT)\right)
		-\lla\tY_{t,x};\frac{\td\mv}{\mr}\rra\cdot\nabla_x\left(p(\mr,\mT)\right)\right]dxdt\nonumber\\
		&-\dip\int\limits_{\si}^{\tau}\int\limits_{\Omega}\left[\frac{\pa_{\tT}(H_{\mT}(\mr,\mT)}{\mr}
		\left(\lla\tY_{t,x}; \td\rra\pa_t\mT+\langle\tY_{t,x};\tm\rangle\cdot\nabla_x\mT\right)+\lla\tY_{t,x} ;p( \td,\bar{\Theta})\rra\mbox{div}_x\mv\right]dxdt,\label{def:J2}
	\end{align}
	and $\mathcal{J}^\e_3,\mathcal{R}(\si,\tau)$ are defined as
	\begin{align}
	\mathcal{J}^\e_3	&:=	\dip\int\limits_{\si}^{\tau}\int\limits_{\Omega}\frac{1}{\mr}\left[
		\lla\tY_{t,x};\td\mv-\tm\rra\cdot\left(\mr\pa_t\mv+\mr\nabla_x\mv\cdot\mv+\nabla_x p(\mr,\mT)\right)\right]dxdt,\label{def:J3}\\
\mathcal{R}(\si,\tau)&:=\dip\int\limits_{\si}^{\tau}\left[\int\limits_\Omega\int\limits_{S^{N-1}}((\xi\otimes\xi):\nabla_x\mathfrak{u}_\e) d\tDcon(t)\right]dt+(\g-1)\int\limits_{\si}^{\tau}\left[\int\limits_\Omega (\mbox{div}_x\mathfrak{u}_\e) d\tDint(t)\right]dt.\label{def:R}
\end{align}
As we have mentioned before that the variables $\rho$, $\Theta$ in $[r,R]$ and $[\underline{\theta},\bar{\theta}]$ respectively. Now it is clear from mean value inequality that  
\begin{equation}\label{coer}
\mathfrak{E}\left(\td,\tm,\bar{\Theta}\left|\right.\mr,\mv,\mT\right)\geq C_2\left( \frac{1}{2}\bar{\rho}\left|\frac{\tm}{\td}-\mv\right|^2+|\td-\mr|^2+|\tS-\mathfrak{S}_\e|^2\right).
\end{equation}
The constant $C_2$ depends on $r,R,\underline{\theta},\bar{\theta}$. Proof of the inequality (\ref{coer}) could be more involved in general. For instance, if we just assume $\rho>0$ and $\tS>s_0$ then the proof may not follow easily by mean value inequality. We refer interested reader to \cite{FeiVas} for the general situation. By the assumption of one sided Lipschitz condition on $\mathfrak{u}$ and \eqref{coer} we get the following estimates of $\mathcal{J}^\e_1$ and $\mathcal{R}(\si,\tau)$,
	\begin{align}
	\mathcal{J}_1^\e&\leq\domm \mathcal{C}(t)\lla\tY_{t,x};\mathfrak{E}\left(\td,\tm,\bar{\Theta}\left|\right.
	\mr,\mv,\mT\right)\rra \,dxdt,\label{J1}\\
	-\mathcal{R}(\si,\tau)&\leq	\domm \mathcal{C}(t) \left(d\tDkin(t)+d\tDint(t)\right)\, dxdt.\label{R}
	\end{align} 
  Our next aim is to estimate the term $\mathcal{J}_2^\e$ defined as in \eqref{def:J2}. We estimate $\mathcal{J}_3^\e$ in section \ref{commutator}. 
	Before we  analyze the term $\mathcal{J}^\e_2$ we want to make a simple observation in the following claim.
	\begin{claim}\label{cl1} Let $F\in L^{\f}([0,T)\times\Omega)$. Then we have
		\begin{equation}\label{P1}
		\begin{array}{lll}
		-\dip \domm\pa_\tT p(\mr,\mT)\mbox{div}_x\mv Fdxdt
		&=&\dip\domm\mr\pa_\tT s(\mr,\mT)\left(\pa_t \mT+\nabla_x \Theta\cdot\mv\right)Fdxdt\\
		&+&\dip\mathcal{Q}_1^{\e}+\mathcal{Q}_2^{\e}+\mathcal{Q}_3^{\e}+\mathcal{Q}_4^{\e}
		\end{array}\end{equation}
		where $\mathcal{Q}_1^{\e},\mathcal{Q}_2^{\e},\mathcal{Q}_3^{\e},\mathcal{Q}_4^{\e}$ are defined as follows
		\begin{eqnarray*}
			\mathcal{Q}_1^{\e}&:=&\domm\left[\pa_t(\tr s(\tr,\tT))_\e-\pa_t(\mr s(\mr,\mT))\right]Fdxdt,\\
			\mathcal{Q}_2^{\e}&:=&\domm\left[\mbox{div}_x(\tr s(\tr,\tT)\mathfrak{u})_\e-\mbox{div}_x(\mr s(\mr,\mT)\mv)\right]Fdxdt,\\
			\mathcal{Q}_3^{\e}&:=&\domm s(\mr,\mT)\left(\mbox{div}_x(\mr\mv)-\mbox{div}_x(\tr\mathfrak{u})_\e\right)Fdxdt,	
		\end{eqnarray*}
	\begin{eqnarray*}			
			\mathcal{Q}_4^{\e}&:=&\domm\frac{1}{\mr}\pa_\tT p(\mr,\tT_\e)\left(\mbox{div}_x(\tr\mathfrak{u})_\e-\mbox{div}_x(\mr\mv)\right)Fdxdt.
		\end{eqnarray*}
	\end{claim}
	This can be proved by mollifying the Euler system (\ref{I1})--(\ref{I3}) and the entropy equality for the trio $[\tr,\mathfrak{u},\tT]$. To make the presentation brief we skip the proof of Claim \ref{cl1}. 	Identities (\ref{P1}),(\ref{P2}) imply
	\begin{eqnarray*}
		\mathcal{J}^\e_2 &=&\dip\domm\left[\lla\tY_{t,x};\td s(\mr,\mT)-\td s(\td,\bar{\Theta})\rra\pa_t\mT+\lla\tY_{t,x}; s(\mr,\mT)\tm-s(\td,\bar{\Theta})\tm\rra\cdot\nabla_x\mT\right]dxdt\\
		&+ &\dip\int\limits_{\si}^{\tau}\int\limits_{\Omega}\left[
		\lla\tY_{t,x};1-\frac{\td}{\mr}\rra
		\pa_{\tr}p(\mr,\mT)\left(\pa_t\mr+\mv\cdot\nabla_x\mr\right)
		-\mv\cdot\nabla_x\left(p(\mr,\mT)\right)\right]dxdt\\
		&+ &\dip\int\limits_{\si}^{\tau}\int\limits_{\Omega}\left[
		\lla\tY_{t,x};1-\frac{\td}{\mr}\rra
		\pa_{\tT}p(\mr,\mT)\left(\pa_t\mT+\mv\cdot\nabla_x\mT\right)\right]dxdt\\
		&-&\dip\int\limits_{\si}^{\tau}\int\limits_{\Omega}\left[\lla\tY_{t,x} ;p( \td,\bar{\Theta})-\frac{(\g-1)\mT}{\mr}(\td s(\td,\bar{\Theta})-\mr s(\mr,\mT))\pa_\tT p(\mr,\mT)\rra\mbox{div}_x\mv \right]dxdt\\
		&+&\dip\domm(\g-1)\mT(\td s(\td,\bar{\Theta})-\mr s(\mr,\mT))\pa_\tT s(\mr,\mT)\left(\pa_t \mT+\nabla_x \Theta\cdot\mv\right)dxdt\\
		&+&\dip\mathcal{Q}_1^{\e}+\mathcal{Q}_2^{\e}+\mathcal{Q}_3^{\e}+\mathcal{Q}_4^{\e}.		
	\end{eqnarray*}
 We use integration by parts to get 
 \begin{equation*}
 -\domm\mv\cdot\nabla_x p(\mr,\mT)dxdt=\domm \mbox{div}_x(\mv)p(\mr,\mT)dxdt.
 \end{equation*}
 Consequently, we write $\mathcal{J}_2^\e$ as follows
 \begin{equation}
 	\mathcal{J}^\e_2=\mathcal{K}_1^\e+\mathcal{K}_2^\e+\mathcal{K}_3^\e+\mathcal{K}_4^\e+\dip\mathcal{Q}_1^{\e}+\mathcal{Q}_2^{\e}+\mathcal{Q}_3^{\e}+\mathcal{Q}_4^{\e},
 \end{equation}
where $\mathcal{K}^\e_i$'s are defined as follows
	\begin{align*}
	\mathcal{K}^\e_1 &:=\dip\domm\left[\lla\tY_{t,x};\td s(\mr,\mT)-\td s(\td,\bar{\Theta})\rra\pa_t\mT+\lla\tY_{t,x}; s(\mr,\mT)\tm-s(\td,\bar{\Theta})\tm\rra\cdot\nabla_x\mT\right]dxdt,\\
\mathcal{K}^\e_2 &:=\dip\int\limits_{\si}^{\tau}\int\limits_{\Omega}\left[
	\lla\tY_{t,x};1-\frac{\td}{\mr}\rra
	\pa_{\tr}p(\mr,\mT)\left(\mbox{div}_x(\mr\mv)-\mbox{div}_x(\tr\mathfrak{u})_\e\right)\right]dxdt,\\		
\mathcal{K}^\e_3 &:=\dip\int\limits_{\si}^{\tau}\int\limits_{\Omega}\left[
	\lla\tY_{t,x};
	\left(\td-\mr\right)\pa_{\tr}s(\mr,\mT)\rra
	\left(\pa_t\mT+\mv\cdot\nabla_x\mT\right)\right]dxdt\\
	&+ \dip\int\limits_{\si}^{\tau}\int\limits_{\Omega}\left[
	\lla\tY_{t,x};(\g-1)\mT\left(\td s(\td,\bar{\Theta})-\mr s(\mr,\mT)\right)\pa_{\tT}s(\mr,\mT)\rra
	\left(\pa_t\mT+\mv\cdot\nabla_x\mT\right)\right]dxdt,
		\end{align*}
	\begin{align*}
\mathcal{K}^\e_4 &:=-\dip\int\limits_{\si}^{\tau}\int\limits_{\Omega}\left[\lla\tY_{t,x} ;p(\td,\bar{\Theta})-(\td-\mr)\pa_{\tr}p(\mr,\mT)\rra\mbox{div}_x\mv\right]dxdt\\
	&+\dip\int\limits_{\si}^{\tau}\int\limits_{\Omega}\left[\lla\tY_{t,x};
	\frac{(\g-1)\mT}{\mr}(\td s(\td,\bar{\Theta})-\mr s(\mr,\mT))\pa_{\tT}p(\mr,\mT)
	+p(\mr,\mT)\rra\mbox{div}_x\mv\right]dxdt.
\end{align*}
	Our goal is to estimate $\mathcal{K}^\e_i$'s and $\mathcal{Q}^\e_i$'s. We estimate  $\mathcal{Q}^\e_i$ for $1\leq i\leq4$ and $\mathcal{K}_2^\e$ in section \ref{commutator}. Since $\mathfrak{T}\in W^{1,\f}((\de,T)\times\Omega)$ for all $\de>0$, by using \eqref{coer} we bound $\mathcal{K}_1^\e,\mathcal{K}_3^\e$ as follows,
	\begin{align}
	\mathcal{K}_1^\e&\leq C_1 \domm \lla\tY_{t,x};\mathfrak{E}\left(\td,\tm,\bar{\Theta}\left|\right.
	\mr,\mv,\mT\right)\rra dxdt,\label{K1}\\
	\mathcal{K}_3^\e&\leq C_3\domm \lla\tY_{t,x};\mathfrak{E}\left(\td,\tm,\bar{\Theta}\left|\right.
	\mr,\mv,\mT\right)\rra dxdt.\label{K3}
	\end{align} 
	The constants $C_1,C_3$ depend on $\|\tT\|_{W^{1,\f}((\si,\tau)\times\Omega)}$. 
	To bound $\mathcal{K}_4^\e$, we consider the following 
	\begin{eqnarray*}
& &p(\td,\bar{\Theta})-(\td-\mr)\pa_{\tr}p(\mr,\mT)-\frac{(\g-1)\mT}{\mr}(\td s(\td,\bar{\Theta})-\mr s(\mr,\mT))\pa_{\tT}p(\mr,\mT)-p(\mr,\mT)\\
&=&\tilde{p}(\td,\tS)-(\td-\mr)\pa_{\tr}\tilde{p}(\mr,\mathfrak{S}_\e)-(\tS-\mathfrak{S}_\e)\pa_{\mathfrak{S}}\tilde{p}(\mr,\mathfrak{S}_\e)-\tilde{p}(\mr,\mathfrak{S}_\e)
	\end{eqnarray*}
	where $\tS:=\td s(\td,\bar{\Theta})$, $\mathfrak{S}_\e:=\mr s(\mr,\mT)$ and $\tilde{p}:\re_+\times\re_+\rr\re$ is defined as
	\begin{equation*}
	\tilde{p}(\rho,\tSs):=p\left(\rho,\tds^{\g-1}\mbox{exp}\left\{(\g-1)\frac{\tSs}{\tds}\right\}\right).
	\end{equation*}	
	By the assumption taken on pressure $p$ it can be shown that $(\rho,\tSs)\mapsto \tilde{p}(\rho,\tSs)$ is a convex function, that is, the Hessian matrix is non-negative definite (see Lemma 3.1 on page-10 of \cite{BFH}). Hence we get
		\begin{eqnarray}\label{44}
	\mathcal{K}_4^\e&\leq& C_4\domm\mathcal{C}(t) \lla\tY_{t,x};|\td-\mr|^2+|\tS-\mathfrak{S}_\e|^2\rra dxdt.
	\end{eqnarray}
 From (\ref{44}) and (\ref{coer}) we have 
		\begin{equation}\label{K4}
		\mathcal{K}_4^\e\leq C_4\domm \mathcal{C}(t)\lla\tY_{t,x};\mathfrak{E}\left(\td,\tm,\bar{\Theta}\left|\right.
		\mr,\mv,\mT\right)\rra dxdt
		\end{equation}
	where the constant $C_4$ depends on $r,R,\underline{\theta},\bar{\theta}$ and $\|p(\cdot,\cdot)\|_{C^2}$. We proceed to the next section which is devoted to estimate $\mathcal{J}^\e_3$ and $\mathcal{K}_2^\e$. 
	\subsection{Commutator estimate}\label{commutator}
	\begin{lemma}\label{comm} 
		
		Let $\mathcal{U}\subset\R^M$ be a bounded open set. Let $\tilde{ \mathcal{U}} \subset \re^M$ be another bounded open set such that $\Ov{\mathcal{U}}\subset \tilde{\mathcal{U}}$. Suppose that $\mathfrak{F}: \tilde {\mathcal{U}} \to \re^k$ is defined as $\mathfrak{F}=(\mathfrak{f}_1,\cdots,\mathfrak{f}_k)$ such that for each $i\in\{1,\cdots,k\}$ the component $\mathfrak{f}_i$ belongs to the Besov space 
		$B^{\alpha_i, \infty}_p(\tilde{\mathcal{U}}, \re)$, $p \geq 2$. Let $\eta_\ep$ be Friedrich mollifiers such that $\mbox{ supp}[ \eta_{\e} ] \subset \{ |y| < \e \}$. Let $G\in C^2(K,\R)$ where $K$ is a convex open set containing image of $\mathfrak{F}$.
		 Then we have
		\begin{equation}
		\left\| \nabla_y G( \mathfrak{F}_\e ) - \nabla_y  G(\mathfrak{F})_\e \right\|_{L^{\frac{p}{2}} (\mathcal{U}; \re^M) }
		\leq \sum\limits_{|(\g_1,\cdots,\g_k)=\g|=2}\e^{\sum\limits_{j=1}^{k}\g_j\al_j-1}(\sup|\pa^\g G|)\left(\prod\limits_{j=1}^{k}|\mathfrak{f}_j|^{\g_j}_{B^{\al_j,\f}_p(\tilde{\mathcal{U}})}\right)
		\end{equation} 
		for $\nabla_y = (\partial_{y_1}, \dots, \partial_{y_M})$.
	\end{lemma} 
	One can prove this lemma by using Taylor expansion of $C^2$ function and estimate of difference quotient for Besov function.  For a sake of completeness we give a sketch of the proof in the Appendix. Next we estimate $\mathcal{J}_3^\e$ and $\mathcal{K}_2^\e$ by employing Lemma \ref{comm}. To this end we mollify the continuity equation and momentum equation to get 
	\begin{align*}
		\pa_t\tr_\e+\mbox{div}_x(\tr\mathfrak{u})_\e&=0,\\
		\pa_t(\tr\mathfrak{u})_\e+\mbox{div}_x(\tr\mathfrak{u}\otimes\mathfrak{u})_\e+\nabla_xp(\tr,\tT)_\e&=0.
	\end{align*}
	Therefore we write    
	\begin{align*}
		\mathcal{J}^\e_3
		&=	\dip\int\limits_{\si}^{\tau}\int\limits_{\Omega}\frac{1}{\mr}\left[
		\lla\tY_{t,x};\td\mv-\tm\rra\cdot
		\left(\pa_t(\mr\mv)+\nabla_x\mv\cdot(\mr\mv)+\mbox{div}_x(\mr\mv)\mv+\nabla_x p(\mr,\mT)\right)\right]dxdt\\
		&-\domm \frac{1}{\mr} \lla\tY_{t,x};\td\mv-\tm\rra\cdot\mv\left(\pa_{t}\mr+\mbox{div}_x(\mr\mv)\right)dxdt\\
		&=\domm\frac{1}{\mr}\lla\tY_{t,x};\td\mv-\tm\rra\cdot\left[\left(\pa_t(\mr\mv)-\pa_t(\tr\mathfrak{u})_\e\right)+\left(\mbox{div}_x(\mr\mv\otimes\mv)-\mbox{div}_x(\tr\mathfrak{u}\otimes\mathfrak{u})\right)\right]dxdt\\
		&+\domm \frac{1}{\mr} \lla\tY_{t,x};\td\mv-\tm\rra\cdot\left[\mv\left(\mbox{div}_x(\tr\mathfrak{u})_\e-\mbox{div}_x(\mr\mv)\right)+\left(\nabla_x p(\mr,\mT)-\nabla_x p(\tr,\tT)_\e\right)\right]dxdt.
	\end{align*}
	By Lemma \ref{comm} we have
	\begin{equation}\label{J3}
	\mathcal{J}^\e_3\leq\left(\e^{2\B-1}+ \e^{\al+\B-1}\right)C(\|\tr\|_{B^{\al,\f}_{q}},\|\mathfrak{u}\|_{B^{\B,\f}_{q}}).
	\end{equation}
	Similarly, we have
	\begin{equation}\label{K2}
	\mathcal{K}^\e_2+\dip\mathcal{Q}_1^{\e}+\mathcal{Q}_2^{\e}+\mathcal{Q}_3^{\e}+\mathcal{Q}_4^{\e}
	\leq \e^{\al+\B-1}C(\|\tr\|_{B^{\al,\f}_{q}},\|\mathfrak{u}\|_{B^{\B,\f}_{q}}).
	\end{equation}
	\subsection{Passage of limit $\e\rr0$  and a Gronwall type argument}
	By clubbing (\ref{J1}), (\ref{R}), (\ref{K1}), (\ref{K2}), (\ref{K3}), (\ref{K4}) and (\ref{J3}) we have
	\begin{align*}
		\dip\left[\int\limits_{\Omega}\lla\tY_{t,x};\mathfrak{E}\left(\td,\tm,\bar{\Theta}\left|\right.
		\mr,\mv,\mT\right)\rra dx\right]_{t=\si}^{t=\tau}&\leq C\left(\e^{2\B-1}+\e^{\al+\B-1}\right)\\
		&+\domm K(t)\lla\tY_{t,x};\mathfrak{E}\left(\td,\tm,\bar{\Theta}\left|\right.
		\mr,\mv,\mT\right)\rra dxdt 
	\end{align*}
	where $C=C(\|\tr\|_{B^{\al,\f}_{q}},\|\mathfrak{u}\|_{B^{\B,\f}_{q}})$ and $K(t)=C(t)+C_1+C_3+C_4$. Because of the regularity assumption \ref{b} on $(\mathfrak{r},\mathfrak{u},\mathfrak{T})$ we are all set to pass the limit $\e\rr0$ and a sequence of $\si_k\rr0$ to conclude the Theorem \ref{theorem1}.  This completes the proof Theorem \ref{theorem1}.\qed
	\begin{remark}
	 We can also prove the similar result without assuming the caloric state equation and Boyle's law then we need to compensate by assuming convexity of the pressure $p$.
	\end{remark}

\section{Proof of Theorem \ref{theorem2}}\label{proof2}
In this section, we prove Theorem \ref{theorem2}. Proof is almost similar to Theorem \ref{theorem1}. Note that in this situation, admissible weak solution possesses Besov regularity $B^{\al,\f}_{3}$ for $\al>1/3$. We mollify the system \eqref{I1}--\eqref{I3} for both solutions. Therefore, we can not invoke Proposition \ref{proposition1} directly. We prove an analogous version of \eqref{rel-en-es} with some remainder terms which depend on mollifying parameter $\e>0$. Rest of the proof is similar to Theorem \ref{theorem1} and passage of limit $\e\rr0$ relies on a commutator estimate which is closer to commutator estimates of \cite{CET,Drivas,FeGwGwWi,GwMiGw} rather than Lemma \ref{comm}.

\begin{proof}[Proof of Theorem \ref{theorem2}]
Since $(\rho,\textbf{u},\vartheta)$ satisfies \eqref{I1}--\eqref{I4} in the sense of distribution, we mollify the system \eqref{I1}--\eqref{I2} and \eqref{I4} to get
\begin{eqnarray}
\pa_t\rho_\e+\mbox{div}_x(\rho\textbf{u})_\e&=&0,\label{eqn1}\\
\pa_t(\rho\textbf{u})_\e+\mbox{div}_x(\rho\textbf{u}\otimes\textbf{u})_\e+\nabla_xp(\rho,\vt)_\e&=&0,\label{eqn2}\\
\pa_t(\rho s(\rho,\vt))_\e+\mbox{div}_x(\rho s(\rho,\vt)\textbf{u})_\e&\geq&0.\label{eqn4}
\end{eqnarray}
We modify \eqref{eqn1}--\eqref{eqn4} to obtain
\begin{align}
\pa_t\rho_\e+\mbox{div}_x(\rho_\e\textbf{u}_\e)&=R_1^{\e},\label{mollified:cont1}\\
\pa_t(\rho_\e\textbf{u}_\e)+\mbox{div}_x(\rho_\e\textbf{u}_\e\otimes\textbf{u}_\e)+\nabla_xp(\rho_\e,\vt_\e)&=R_2^{\e},\label{mollified:mom1}\\
\pa_t(\rho_\e s(\rho_\e,\vt_\e))+\mbox{div}_x(\rho_\e s(\rho_\e,\vt_\e)\textbf{u}_\e)&\geq R_3^{\e}\label{eq:mollified:s},
\end{align}
where remainder terms $R_1^\e,R_2^\e$ and $R_3^\e$ are defined as
\begin{align}
R_1&:=\mbox{div}_x(\rho_\e\textbf{u}_\e)-\mbox{div}_x(\rho\textbf{u})_\e,\label{def:R1}\\
R_2^{\e}&:=\pa_t(\rho_\e\textbf{u}_\e)-\pa_t(\rho\textbf{u})_\e+\mbox{div}_x(\rho_\e\textbf{u}_\e\otimes\textbf{u}_\e)-\mbox{div}_x(\rho\textbf{u}\otimes\textbf{u})_\e,+\nabla_xp(\rho_\e,\vt_\e)-\nabla_xp(\rho,\vt)_\e,\label{def:R2}\\
R_3^{\e}&:=\pa_t(\rho_\e s(\rho_\e,\vt_\e))-\pa_t(\rho s(\rho,\vt))_\e+\mbox{div}_x(\rho_\e s(\rho_\e,\vt_\e)\textbf{u}_\e)-\mbox{div}_x(\rho s(\rho,\vt)\textbf{u})_\e.\label{err:R3}
\end{align}
%
%
%
Since $(\rho,\textbf{u},\vartheta)$ satisfies \eqref{I3}, by using appropriate test function (see for instance, \cite[Section 5]{Dafermos_book}) one can show
\begin{equation}\label{eq:energy}
\left[\int\limits_{\Omega}\frac{1}{2}\rho\abs{\textbf{u}}^2+\rho e(\rho,\vartheta)\,dx\right]_{t=\si}^{t=\tau}=0\mbox{ for }0\leq\si<\tau\leq T.
\end{equation}
By our assumption, $(r,\textbf{v},\theta)$ also satisfies \eqref{I1}--\eqref{I3} in the sense of distribution. Therefore, we similarly have 
\begin{align}
\pa_tr_\e+\mbox{div}_x(r_\e\textbf{v}_{\e})&=R_4^{\e},\label{mollified:eqn:cont}\\
\pa_t(r_\e\textbf{v}_{\e})+\mbox{div}_x(r_\e\textbf{v}_{\e}\otimes\textbf{v}_{\e})+\nabla_xp(r_\e,\theta_\e)&=R_5^{\e},\label{mollified:eqn:mom}
\end{align}
where $R_4^{\e}$ and $R_5^{\e}$ are defined as follows
\begin{align}
R_4^{\e}&:=\mbox{div}_x(r_\e\textbf{v}_{\e})-\mbox{div}_x(r\textbf{v})_\e,\label{def:R4}\\
R_5^{\e}&:=\pa_t(r_\e\textbf{v}_{\e}-(r\textbf{v})_\e)+\mbox{div}_x((r_\e\textbf{v}_{\e}\otimes\textbf{v}_{\e})-(r\textbf{v}\otimes\textbf{v})_\e)+\nabla_x(p(r_\e,\theta_\e)-p(r,\theta)_\e).\label{def:R5}
\end{align}
Modifying \eqref{mollified:eqn:cont} and \eqref{mollified:eqn:mom} we obtain 
\begin{align}
\pa_tr_\e&=-\textbf{v}_{\e}\cdot\nabla_xr_\e-r_\e\mbox{div}_x\textbf{v}_{\e}+R_4^\e,\label{eqn:partial_r}\\
\pa_t \textbf{v}_{\e}&=-\textbf{v}_{\e}\cdot\nabla_x\textbf{v}_{\e}-\frac{1}{r_\e}\nabla_x p(r_\e,\theta_\e)-R_4^{\e}\textbf{v}_{\e}+\frac{1}{r_\e}R_5^{\e}.\label{eqn:partial_v}
\end{align}
We integrate \eqref{mollified:mom1} against $\textbf{v}_{\e}$ to have
\begin{equation*}
\domm\pa_t(\rho_\e\textbf{u}_\e)\cdot\textbf{v}_{\e}+\mbox{div}_x(\rho_\e\textbf{u}_\e\otimes\textbf{u}_\e)\cdot\textbf{v}_{\e}+\nabla_xp(\rho_\e,\vt_\e)\cdot\textbf{v}_{\e}\,dxdt=\domm R_2^{\e}\cdot\textbf{v}_{\e}\,dxdt.
\end{equation*}
Integrating by parts we obtain
\begin{equation}\label{rel_ent:C1}
\left[\int\limits_{\Omega}\rho_\e\textbf{u}_\e\cdot\textbf{v}_{\e}\,dx\right]_{t=\si}^{t=\tau}=\domm\rho_\e\textbf{u}_\e\cdot\pa_t\textbf{v}_{\e}+\nabla_x\textbf{v}_{\e}:(\rho_\e\textbf{u}_\e\otimes\textbf{u}_\e)+p(\rho_\e,\vt_\e)\mbox{div}_x\textbf{v}_{\e}\,dxdt+\domm R_2^{\e}\cdot\textbf{v}_{\e}\,dxdt.
\end{equation}
Next we integrate \eqref{mollified:mom1} against $\frac{1}{2}\abs{\textbf{v}_{\e}}^2$ to get
\begin{equation*}
\frac{1}{2}\domm \pa_t\rho_\e\abs{\textbf{v}_{\e}}^2+\mbox{div}_x(\rho_\e\textbf{u}_\e)\abs{\textbf{v}_{\e}}^2\,dxdt=\frac{1}{2}\domm R_1^{\e}\abs{\textbf{v}_{\e}}^2\,dxdt.
\end{equation*}
Now we do an integration by parts to have
\begin{equation}\label{rel_ent:C2}
\left[\int\limits_{\Omega}\frac{1}{2}\rho_\e\abs{\textbf{v}_{\e}}^2\,dx\right]_{t=\si}^{t=\tau}=\domm \rho_\e\textbf{v}_{\e}\cdot\pa_t\textbf{v}_{\e}+\nabla_x\textbf{v}_{\e}:(\rho_\e\textbf{u}_\e\otimes\textbf{v}_{\e})\,dxdt+\frac{1}{2}\domm R_1^{\e}\abs{\textbf{v}_{\e}}^2\,dxdt.
\end{equation}
Combining \eqref{eq:energy}, \eqref{rel_ent:C1} and \eqref{rel_ent:C2} we obtain
\begin{align}
&\left[\int\limits_{\Omega}\frac{1}{2}\rho\abs{\textbf{u}}^2+\rho e(\rho,\vartheta)-\rho_\e\textbf{u}_\e\cdot\textbf{v}_{\e}+\frac{1}{2}\rho_\e\abs{\textbf{v}_{\e}}^2\,dx\right]_{t=\si}^{t=\tau}\nonumber\\
&=\domm \rho_\e(\textbf{v}_{\e}-\textbf{u}_\e)\cdot\pa_t\textbf{v}_{\e}+\rho_\e\nabla_x\textbf{v}_{\e}:(\textbf{u}_\e\otimes(\textbf{v}_{\e}-\textbf{u}_\e))-p(\rho_\e,\vt_\e)\mbox{div}_x\textbf{v}_{\e}\,dxdt\nonumber\\
&+\domm \frac{1}{2}R_1^{\e}\abs{\textbf{v}_{\e}}^2-R_2^{\e}\cdot\textbf{v}_{\e}\,dxdt.\label{rel_ent:C3}
\end{align}
Using \eqref{eqn:partial_v} in \eqref{rel_ent:C3} we get
\begin{align}
&\left[\int\limits_{\Omega}\frac{1}{2}\rho\abs{\textbf{u}}^2+\rho e(\rho,\vartheta)-\rho_\e\textbf{u}_\e\cdot\textbf{v}_{\e}+\frac{1}{2}\rho_\e\abs{\textbf{v}_{\e}}^2\,dx\right]_{t=\si}^{t=\tau}\nonumber\\
&=\domm \rho_\e\nabla_x\textbf{v}_{\e}:((\textbf{u}_\e-\textbf{v}_{\e})\otimes(\textbf{v}_{\e}-\textbf{u}_\e))\,dxdt\nonumber\\
&-\domm\frac{\rho_\e}{r_\e}(\textbf{v}_{\e}-\textbf{u}_\e)\cdot(\pa_rp(r_\e,\theta_\e)\nabla_xr_\e+\pa_{\theta_{\e}}p(r_\e,\theta_\e)\nabla_x\theta_\e)
 +p(\rho_\e,\vt_\e)\mbox{div}_x\textbf{v}_{\e}\,dxdt\nonumber\\
&+\domm  \rho_\e(\textbf{v}_{\e}-\textbf{u}_\e)\cdot\left(\frac{1}{r_\e}R_5^{\e}-R_4^{\e}\textbf{v}_{\e}\right)+\frac{1}{2}R_1^{\e}\abs{\textbf{v}_{\e}}^2-R_2^{\e}\cdot\textbf{v}_{\e}\,dxdt.\label{rel_ent:C4}
\end{align}
Next we integrate \eqref{eq:mollified:s} against $\theta_{\e}$ and get
\begin{equation*}
\domm \pa_t(\rho_\e s(\rho_\e,\vt_\e))\theta_{\e}+\mbox{div}_x(\rho_\e s(\rho_\e,\vt_\e)\textbf{u}_\e)\theta_{\e} \,dxdt\geq\domm R_6^{\e}\theta_{\e}\,dxdt.
\end{equation*}
Integrating by parts we obtain
\begin{equation}\label{integral:eqn:theta1}
\left[\int\limits_{\Omega}-\theta_{\e} \rho_\e s(\rho_\e,\vt_\e)\,dx\right]_{t=\si}^{t=\tau}\leq-\domm  \rho_\e s(\rho_\e,\vt_\e)(\pa_t\theta_{\e}+\textbf{u}_\e\cdot\nabla_x\theta_{\e})\,dxdt -\domm R_6^{\e}\theta_{\e}\,dxdt.
\end{equation}
Next we integrate \eqref{mollified:cont1} against $H_{\theta_{\e}}(r_\e,\theta_\e)$ to get
\begin{equation}
\domm \pa_t\rho_\e\pa_r	H_{\theta_{\e}}(r_\e,\theta_\e)+\mbox{div}_x(\rho_\e\textbf{u}_\e)\pa_rH_{\theta_{\e}}(r,\theta_{\e})\,dxdt=\domm R_1^{\e}	H_{\theta_{\e}}(r,\theta_{\e})\,dxdt.
\end{equation}
After an integration by parts we have
\begin{align}
&\left[\int\limits_{\Omega} \rho_\e\pa_r	H_{\theta_{\e}}(r_\e,\theta_{\e})\,dx\right]_{t=\si}^{t=\tau}\nonumber\\
&=\domm \rho_\e\pa_t\pa_r	H_{\theta_\e}(r_\e,\theta_{\e})+\rho_\e\textbf{u}_\e\cdot\nabla_x\pa_rH_{\theta_{\e}}(r_\e,\theta_{\e})+\domm R_1^{\e}	H_{\theta_{\e}}(r_\e,\theta_\e)\,dxdt\nonumber\\
&=\domm \rho_\e\pa_{rr}H_{\theta_\e}(r_\e,\theta_\e)\left(\pa_tr_\e+\textbf{u}_\e\cdot\nabla_x r_\e\right)\,dxdt+\domm \rho_\e\pa_{r\theta_{\e}}H_{\theta_\e}(r_\e,\theta_\e)\left(\pa_t\theta_\e+\textbf{u}_\e\cdot\nabla_x \theta_\e\right)\nonumber\\
&+\domm R_1^{\e}	H_{\theta_\e}(r_\e,\theta_\e)\,dxdt.\label{integral:eq:rhoH1}
\end{align}
Applying \eqref{P2} and \eqref{eqn:partial_r} in \eqref{integral:eq:rhoH1} we obtain
\begin{align}
\left[\int\limits_{\Omega} \rho_\e\pa_r	H_{\theta_{\e}}(r_\e,\theta_{\e})\,dx\right]_{t=\si}^{t=\tau}
&=\domm \frac{\rho_\e}{r_\e}\pa_{r}p(r_\e,\theta_\e)\nabla_x r_\e\cdot(\textbf{u}_\e-\textbf{v}_{\e})-{\rho_\e}\pa_{r}p(r_\e,\theta_\e)\mbox{div}_x\textbf{v}_{\e}\,dxdt\nonumber\\
&-\domm \rho_\e(s(r_\e,\theta_\e)-\pa_{\theta_{\e}}p(r_\e,\theta_\e))\left(\pa_t\theta_\e+\textbf{u}_\e\cdot\nabla_x \theta_\e\right)\,dxdt\nonumber\\
&+\domm R_1^{\e}	H_{\theta_\e}(r_\e,\theta_\e)+\rho_\e\pa_{rr}H_{\theta_\e}(r_\e,\theta_\e)R_4^{\e}\,dxdt.\label{integral:eq:rhoH2}
\end{align}
From \eqref{P15} we have
%
\begin{align}
&\left[\int\limits_{\Omega}{r}_\e\pa_{r} H_{\theta_{\e}}({r}_\e, \theta_{\e})-H_{ \theta_{\e}}({r}_\e, \theta_{\e})\,dx\right]_{t=\si}^{t=\tau}\nonumber\\
&=\domm \pa_r p(r_\e,\theta_{\e})\pa_tr_\e+\pa_\theta p(r_\e\theta_{\e})\pa_t\theta_{\e}\,dxdt\nonumber\\
&=\domm -\pa_r p(r_\e,\theta_{\e})\textbf{v}_{\e}\cdot\nabla_xr_\e-\pa_r p(r_\e,\theta_{\e})r_\e\mbox{div}_x\textbf{v}_{\e}+\pa_\theta p(r_\e,\theta_{\e})\pa_t\theta_{\e}+ \pa_r p(r_\e,\theta_{\e})R_4^\e\,dxdt.\label{integral:eq:H1}
\end{align}
Applying integration by parts we obtain
\begin{align}
&\domm -\pa_r p(r_\e,\theta_{\e})\textbf{v}_{\e}\cdot\nabla_xr_\e\,dxdt\nonumber\\
&=\domm-\textbf{v}_{\e}\cdot\left(\pa_r p(r_\e,\theta_\e)\nabla_xr_\e+\pa_{\theta} p(r_\e,\theta_\e)\nabla_x\theta_\e\right)\,dxdt+\domm \textbf{v}_{\e}\cdot\left(\pa_\theta p(r_\e,\theta_\e)\nabla_x\theta_\e\right)\,dxdt\nonumber\\
&=\domm p(r_\e,\theta_\e)\mbox{div}_\e\textbf{v}_{\e}+\pa_\theta p(r_\e,\theta_\e)\textbf{v}_{\e}\cdot\nabla_x\theta_\e\,dxdt.\label{integral:eq:p1}
\end{align}
Using \eqref{integral:eq:p1} in \eqref{integral:eq:H1} we get
\begin{align}
&\left[\int\limits_{\Omega}\left[{r}_\e\pa_{r} H_{\theta_{\e}}({r}_\e, \theta_{\e})-H_{ \theta_{\e}}({r}_\e, \theta_{\e})\right]dx\right]_{t=\si}^{t=\tau}\nonumber\\
&=\domm (p(r_\e,\theta_\e)-\pa_r p(r_\e,\theta_{\e})r_\e)\mbox{div}_x\textbf{v}_{\e}+\pa_\theta p(r_\e,\theta_{\e})(\pa_t\theta_{\e}+\textbf{v}_{\e}\cdot\nabla_x\theta_\e)\,dxdt
+\domm \pa_r p(r_\e,\theta_{\e})R_4^\e.\label{integral:eq:H2}
\end{align}
Now we invoke Claim \ref{cl1} with $F=\frac{(\g-1)\theta_{\e}}{r_\e}(\rho_\e s(\rho_\e,\vartheta_\e)-r_\e s(r_\e,\theta_{\e}))$ to obtain
\begin{align}
&- \domm\pa_\theta p(r_\e,\theta_\e)\mbox{div}_x\textbf{v}_{\e} \frac{(\g-1)\theta_{\e}}{r_\e}(\rho_\e s(\rho_\e,\vartheta_\e)-r_\e s(r_\e,\theta_{\e}))dxdt\nonumber\\
&=\dip\domm r_\e\pa_{\theta_{\e}} s(r_\e,\theta_\e)\left(\pa_t \theta_\e+\nabla_x \theta_\e\cdot\textbf{v}_{\e}\right)\frac{(\g-1)\theta_{\e}}{r_\e}(\rho_\e s(\rho_\e,\vartheta_\e)-r_\e s(r_\e,\theta_{\e}))\,dxdt+\sum\limits_{k=1}^{4}{Q}_k^{\e}\label{partial_theta:s}
\end{align}
where ${Q}_1^{\e},{Q}_2^{\e},{Q}_3^{\e},{Q}_4^{\e}$ are defined as follows
\begin{align*}
{Q}_1^{\e}&:=\domm\left[\pa_t(rs(r,\theta))_\e-\pa_t(r_{\e} s(r_{\e},\theta_{\e}))\right]\frac{(\g-1)\theta_{\e}}{r_\e}(\rho_\e s(\rho_\e,\vartheta_\e)-r_\e s(r_\e,\theta_{\e}))dxdt,\\
{Q}_2^{\e}&:=\domm\left[\mbox{div}_x(r s(r,\theta)\textbf{v})_\e-\mbox{div}_x(r_{\e} s(r_{\e},\theta_{\e})\textbf{v}_{\e})\right]\frac{(\g-1)\theta_{\e}}{r_\e}(\rho_\e s(\rho_\e,\vartheta_\e)-r_\e s(r_\e,\theta_{\e}))dxdt,\\
{Q}_3^{\e}&:=\domm s(r_{\e},\theta_{\e})\left(\mbox{div}_x(r_{\e}\textbf{v}_{\e})-\mbox{div}_x(r\textbf{v})_\e\right)\frac{(\g-1)\theta_{\e}}{r_\e}(\rho_\e s(\rho_\e,\vartheta_\e)-r_\e s(r_\e,\theta_{\e}))dxdt,
\end{align*}
\begin{align*}
{Q}_4^{\e}&:=\domm\frac{1}{r_{\e}}\pa_\theta p(r_{\e},\theta_{\e})\left(\mbox{div}_x(r\textbf{v})_\e-\mbox{div}_x(r_{\e}\textbf{v}_{\e})\right)\frac{(\g-1)\theta_{\e}}{r_\e}(\rho_\e s(\rho_\e,\vartheta_\e)-r_\e s(r_\e,\theta_{\e}))dxdt.
\end{align*}
Clubbing \eqref{rel_ent:C4}, \eqref{integral:eqn:theta1}, \eqref{integral:eq:rhoH2}, \eqref{integral:eq:H2} and \eqref{partial_theta:s} we have
\begin{equation*}
\left[\int\limits_{\Omega}\mathfrak{E}(\rho_\e,\textbf{u}_\e,\vartheta_\e|r_\e,\textbf{v}_{\e},\theta_\e)+\frac{1}{2}\rho\abs{\textbf{u}}^2-\frac{1}{2}\rho_\e\abs{\textbf{u}_\e}^2+\rho e(\rho,\vartheta)-\rho_\e e(\rho_\e,\vartheta_\e)\,dx\right]_{t=\si}^{t=\tau}\leq I_1+I_2+Q_5^{\e}
\end{equation*}
where $I_1,I_2$ and $Q_5^{\e}$ are defined as follows
\begin{align*}
I_1&:=\domm \rho_\e\nabla_x\textbf{v}_{\e}:((\textbf{u}_\e-\textbf{v}_{\e})\otimes(\textbf{v}_{\e}-\textbf{u}_\e))- [p(\rho_\e,\vt_\e)-\pa_rp(r_\e,\theta_\e)(\rho_\e-r_\e)]\mbox{div}_x\textbf{v}_{\e}\,dxdt\nonumber\\
&+\domm \left[\pa_\theta p(r_\e,\theta_\e)\frac{(\g-1)\theta_{\e}}{r_\e}(\rho_\e s(\rho_\e,\vartheta_\e)-r_\e s(r_\e,\theta_{\e}))+p(r_\e,\theta_\e)\right]\mbox{div}_x\textbf{v}_{\e} \,dxdt,\nonumber\\
I_2&:=-\domm  \rho_\e (s(\rho_\e,\vt_\e)-s(r_\e,\theta_\e))(\pa_t\theta_{\e}+\textbf{u}_\e\cdot\nabla_x\theta_{\e})\,dxdt\nonumber\\
&+\domm \left(r_\e-{\rho_\e}\right)r_\e s(r_\e,\theta_\e)(r,\theta_{\e})(\pa_t\theta_{\e}+\textbf{v}_{\e}\cdot\nabla_x\theta_\e)\,dxdt\nonumber\\
&+\dip\domm r_\e\pa_{\theta_{\e}} s(r_\e,\theta_\e)\left(\pa_t \theta_\e+\nabla_x \theta_\e\cdot\textbf{v}_{\e}\right)\frac{(\g-1)\theta_{\e}}{r_\e}(\rho_\e s(\rho_\e,\vartheta_\e)-r_\e s(r_\e,\theta_{\e}))\,dxdt
\end{align*} and
\begin{align*}
Q_5^{\e}&:=\domm  \rho_\e(\textbf{v}_{\e}-\textbf{u}_\e)\cdot\left(-R_4^{\e}\textbf{v}_{\e}+\frac{1}{r_\e}R_5^{\e}\right)+\frac{1}{2}R_1^{\e}\abs{\textbf{v}_{\e}}^2-R_2^{\e}\cdot\textbf{v}_{\e}-R_6^{\e}\theta_{\e}\,dxdt\nonumber\\
&+{Q}_1^{\e}+{Q}_2^{\e}+{Q}_3^{\e}+{Q}_4^{\e}-\domm R_1^{\e}	H_{\theta_\e}(r_\e,\theta_\e)+\rho_\e\pa_{rr}H_{\theta_\e}(r_\e,\theta_\e)R_4^{\e}+\domm \pa_r p(r,\theta_{\e})R_4^\e\,dxdt.
\end{align*}
Rest of the proof is similar to Theorem \ref{theorem1} except the estimation of $Q_5^{\e}$. Note that $I_1$ and $I_2$ can be estimated in a similar way as we have done for $\mathcal{J}_1^{\e},\mathcal{J}_2^{\e}$ in the proof of Theorem \ref{theorem1}. Hence, we get
\begin{equation}
I_1+I_2\leq \domm \tilde{C}(t)\mathfrak{E}(\rho_\e,\textbf{u}_\e,\vartheta_\e|r_\e,\textbf{v}_{\e},\theta_\e)\,dxdt
\end{equation}
for some $\tilde{C}\in L^1((0,T),[0,\f))$. 
Now we bound $Q^{\e}_5$ by using commutator estimate. It is not same as Theorem \ref{theorem1} or \cite{FGJ}. Here we rely on a different type of commutator estimates mostly used \cite{CET,Drivas,FeGwGwWi,GwMiGw} for proving energy equality or the positive part of Onsager conjecture. To make the presentation brief, we only show the estimation for the term $\domm R_2^{\e}\cdot\textbf{v}_{\e}\,dxdt$. Other terms of $Q_5^{\e}$ can be estimated in a similar way (see also \cite{CET,Drivas,FeGwGwWi,GwMiGw}). Using integration by parts we have
\begin{align*}
-\domm R_2^{\e}\cdot\textbf{v}_{\e}\,dxdt&=\domm (\rho_\e\textbf{u}_\e-(\rho\textbf{u})_\e)\cdot\pa_t\textbf{v}_{\e}+(\rho_\e\textbf{u}_\e\otimes\textbf{u}_\e-(\rho\textbf{u}\otimes\textbf{u})_\e):\nabla_x\textbf{v}_{\e}\,dxdt\\
&+\domm (p(\rho_\e,\vt_\e)-p(\rho,\vt)_\e)\mbox{div}_x\textbf{v}_{\e}\,dxdt+\left[\int\limits_{\Omega}(\rho_\e\textbf{u}_\e-(\rho\textbf{u})_\e)\cdot\textbf{v}_{\e}\,dx\right]_{t=\si}^{t=\tau}.
\end{align*}
Suppose that for $i=1,2,3$, $X_i,Y_i$ are Banach spaces defined as follows
\begin{align*}
&X_1=L^{3/2}((0,T)\times\Omega,\R^N), X_2=L^{3/2}((0,T)\times\Omega,\R^{N\times N}),\,X_3=L^{3/2}((0,T)\times\Omega,\R),\\
&Y_1=L^{3}((0,T)\times\Omega,\R^{N+1}),\, Y_2=L^{3}((0,T)\times\Omega,\R^{2N+1}),\,Y_3=L^{3}((0,T)\times\Omega,\R^{2}).
\end{align*}
By applying \cite[Lemma 3.1]{GwMiGw} we obtain 
\begin{align*}
\|\rho_\e\textbf{u}_\e-(\rho\textbf{u})_\e\|_{X_1}&\leq C_0\left( \|(\rho_\e,\textbf{u}_\e)-(\rho,\textbf{u})\|_{Y_1}^2+\sup\limits_{\tilde{y}\in\R^2,\abs{\tilde{y}}<\e}\|(\rho,\textbf{u})(\cdot-\tilde{y})-(\rho,\textbf{u})\|_{Y_1}^2\right),\\
\|p(\rho_\e,\vt_\e)-p(\rho,\vt)_\e\|_{X_3}&\leq C_0\left( \|(\rho_\e,\theta_\e)-(\rho,\theta_{\e})\|_{Y_3}^2+\sup\limits_{\tilde{y}\in\R^{2},\abs{\tilde{y}}<\e}\|(\rho,\theta_{\e})(\cdot-\tilde{y})-(\rho,\theta_{\e})\|_{Y_3}^2\right),
\end{align*}
and 
\begin{align*}
\|\rho_\e\textbf{u}_\e\otimes\textbf{u}_\e-(\rho\textbf{u}\otimes\textbf{u})_\e\|_{X_2}
&\leq C_0 \|(\rho_\e,\textbf{u}_\e,\textbf{u}_\e)-(\rho,\textbf{u},\textbf{u})\|_{Y_2}^2\\
&+C_0\sup\limits_{\tilde{y}\in\R^2,\abs{\tilde{y}}<\e}\|(\rho,\textbf{u},\textbf{u})(\cdot-\tilde{y})-(\rho,\textbf{u},\textbf{u})\|_{Y_2}^2.
\end{align*}
Now we use regularity assumption \ref{a1}, H\"older inequality and estimates \eqref{Besov_estimate1}--\eqref{Besov_estimate3} for Besov function to have
\begin{equation}
-\domm R_2^{\e}\cdot\textbf{v}_{\e}\,dxdt\leq C_1(\e^{3\B-1}+\e^{\B+1})+\left[\int\limits_{\Omega}(\rho_\e\textbf{u}_\e-(\rho\textbf{u})_\e)\cdot\textbf{v}_{\e}\,dx\right]_{t=\si}^{t=\tau}.
\end{equation}
Note that $(\rho_\e\textbf{u}_\e-(\rho\textbf{u})_\e)\cdot\textbf{v}_{\e}(t,x)\rr0$ as $\e\rr0$ for a.e. $(t,x)\in(0,T)\times\Omega$. Hence we have
\begin{equation}
-\domm R_2^{\e}\cdot\textbf{v}_{\e}\,dxdt\rr0\mbox{ as }\e\rr0.
\end{equation}
Therefore we prove that $Q_5^{\e}\rr0$ as $\e\rr0$. This implies,
\begin{equation}
\left[\int\limits_{\Omega}\mathfrak{E}(\rho,\textbf{u},\vartheta|r,\textbf{v},\theta)\,dx\right]_{t=\si}^{t=\tau}
\leq \domm \tilde{C}(t)\mathfrak{E}(\rho,\textbf{u},\vartheta|r,\textbf{v},\theta)\,dxdt.
\end{equation}
By a similar argument as in the proof of Theorem \ref{theorem1}, we conclude Theorem \ref{theorem2}.
\end{proof}
\section{Application to isentropic Euler system}
We have proved in Theorem \ref{theorem2} that regularity of the Besov solution can be relaxed up to $\al>1/3$ to show uniqueness of weak solution satisfying a one-sided Lipschitz condition. A similar result is true for isentropic Euler system which reads as follows,
\begin{align}
\pa_t\rho+\mbox{div}_x(\rho\textbf{u})&=0\mbox{ for }(t,x)\in(0,T)\times\Omega,\label{eq:isen1}\\
\pa_t(\rho\textbf{u}+\mbox{div}_x(\rho\textbf{u}\otimes\textbf{u})+\nabla_x(\rho^{\g})&=0\mbox{ for }(t,x)\in(0,T)\times\Omega,\label{eq:isen2}
\end{align}
where $\g>1$. Similar to complete Euler system, we work with periodic solutions of \eqref{eq:isen1}--\eqref{eq:isen2} and spatial domain $\Omega$ defined as in \eqref{domain}. For the system \eqref{eq:isen1}--\eqref{eq:isen2}, we prove the following uniqueness result,

	\begin{theorem}\label{theorem:isentropic}
	 Let $(\rho,\textbf{u}),(r,\textbf{v})\in C\left([0,T],L^1(\Omega,\re^{1+N})\right)$ be two weak solutions to the isentropic Euler system \eqref{eq:isen1}--\eqref{eq:isen2} with the initial data $(\rho_0,\textbf{v}_0)$. Moreover, we assume the following properties:
	\begin{enumerate}[(i)]
		\item   $(\rho,\textbf{u}),(r,\textbf{v})$ satisfy the regularity assumption \eqref{alpha-beta1} as in Theorem \ref{theorem2}.
		
		\item Assumption \ref{a} in Theorem \ref{theorem1} holds for $r$ and $\textbf{v}$ with some constants $\underline{r},\overline{r},C>0$. We also assume that $\textbf{v}$ satisfies \eqref{Lip} for some $\mathcal{C}_2\in L^1((0,T),[0,\f))$.  
	\end{enumerate}
	Then the following holds	
	\begin{equation*}
	\rho\equiv r \mbox{ and } \textbf{u}\equiv\textbf{v}\mbox{ for a.e. }(t,x)\in(0,T)\times\Omega.
	\end{equation*} 	
\end{theorem}
Proof of Theorem \ref{theorem:isentropic} follows from a suitable adaptation of the method in section \ref{proof2} with appropriate choice of relative entropy. To make the presentation brief we omit the proof.  Suitable choice of relative entropy for the isentropic Euler system \eqref{eq:isen1}--\eqref{eq:isen2} can be found in \cite{FGJ,FeiKre}.

Note that in order to prove the uniqueness of the solution satisfying \eqref{Lip}, Theorem \ref{theorem:isentropic} requires less regularity than \cite[Theorem 2.1]{FGJ}. But on the same point, Theorem \ref{theorem:isentropic} proves the uniqueness among the class of Besov solutions whereas \cite[Theorem 2.1]{FGJ} gives uniqueness within weak solutions obeying energy inequality.

\vspace*{-.25cm}

	\section*{Appendix}
	
	\begin{proof}[Proof of Lemma \ref{comm}]
		We write 
		\begin{equation*}
		\nabla_y G( \mathfrak{F}_\e ) - \nabla_y  G(\mathfrak{F})_\e
		=DG(\mathfrak{F}_\e)\nabla_y\mathfrak{F}_\e-DG(\mathfrak{F})\nabla_y\mathfrak{F}_\e+DG(\mathfrak{F})\nabla_y\mathfrak{F}_\e- \nabla_y  G(\mathfrak{F})_\e.
		\end{equation*}
		Next we estimate 
		\begin{align*}
			\|DG(\mathfrak{F}_\e)\nabla_y\mathfrak{F}_\e-DG(\mathfrak{F})\nabla_y\mathfrak{F}_\e\|_{L^\frac{p}{2}}&=\left\|\sum\limits_{i=1}^{k}\left[\pa_iG(\mathfrak{F}_\e)-\pa_iG(\mathfrak{F})\right]\nabla_y\mathfrak{f}_i\right\|_{L^\frac{p}{2}}\\
			&=\left\|\sum\limits_{i=1}^{k}\sum\limits_{j=1}^{k}\pa_{ij}G(\mathfrak{F}_{ij}^*)\left[(\mathfrak{f}_j)_\e-\mathfrak{f}\right]\nabla_y\mathfrak{f}_i\right\|_{L^\frac{p}{2}}\\
			&\leq\sum\limits_{|(\g_1,\cdots,\g_k)=\g|=2}\e^{\sum\limits_{j=1}^{k}\g_j\al_j-1}(\sup|\pa^\g G|)\left(\prod\limits_{j=1}^{k}|\mathfrak{f}_j|^{\g_j}_{B^{\al_j,\f}_p}\right).
		\end{align*}
		For the other part we do the following
		\begin{align*}
			& DG(\mathfrak{F})\nabla_y\mathfrak{F}_\e- \nabla_y  G(\mathfrak{F})_\e\\
			&=\sum\limits_{i=1}^{k}\left[\pa_iG(\mathfrak{F})\nabla_y(\mathfrak{f}_i*\eta_\e)\right]-\nabla_y(G(\mathfrak{F})*\eta_\e)
			\\
			&=\int\limits_{\tilde{\mathcal{U}}}\left( \sum\limits_{i=1}^{k}\left[\pa_iG(\mathfrak{F}(x))\mathfrak{f}_i(x-y)\nabla_y\eta_\e(y)\right]-G(\mathfrak{F}(x-y))\nabla_y\eta_\e(y)\right)dy\\
			&=\int\limits_{\tilde{\mathcal{U}}} \left[G(\mathfrak{F}(x))-\sum\limits_{i=1}^{k}\pa_iG(\mathfrak{F}(x))\left(\mathfrak{f}_i(x)-\mathfrak{f}_i(x-y)\right)-G(\mathfrak{F}(x-y))\right]\nabla_y\eta_\e(y)dy\\
			&=\int\limits_{\tilde{\mathcal{U}}}\sum\limits_{i,j=1}^{k}\pa_{ij}G(\mathfrak{F}^*_{ij}(x,y))\left(\mathfrak{f}_i(x)-\mathfrak{f}_i(x-y)\right)\left(\mathfrak{f}_j(x)-\mathfrak{f}_j(x-y)\right)\nabla_y\eta_\e(y)dy.
		\end{align*}
		Therefore we obtain 
		\begin{align*}
			\|DG(\mathfrak{F})\nabla_y\mathfrak{F}_\e- \nabla_y  G(\mathfrak{F})_\e\|_{L^\frac{p}{2}}
			&\leq\sum\limits_{|(\g_1,\cdots,\g_k)=\g|=2}\e^{\sum\limits_{j=1}^{k}\g_j\al_j-1}(\sup|\pa^\g G|)\left(\prod\limits_{j=1}^{k}|\mathfrak{f}_j|^{\g_j}_{B^{\al_j,\f}_p}\right).
		\end{align*}
		This completes the proof of Lemma \ref{comm}.
	\end{proof}

	\noindent\textbf{Acknowledgement.} We acknowledge the support of the Department of Atomic Energy, Government of India, under project no. 12-R\&D-TFR-5.01-0520. The first author would like to thank Inspire faculty-research grant DST/INSPIRE/04/2016/000237.
	Authors are very thankful to Eduard Feireisl for enlightening discussions.
	\vspace*{.25cm}

	\noindent\textbf{Disclosure statement:}
	On behalf of all authors, the corresponding author states that there is no conflict of interest.

\end{document}